\documentclass{amsart}
\pdfoutput=1

\usepackage{amsmath,amssymb,latexsym}
\usepackage{mathtools}
\usepackage[utf8]{inputenc}
\usepackage{units}
\usepackage{enumitem}
\usepackage{array}
\usepackage{booktabs}

\usepackage{placeins}
\usepackage[obeyspaces,hyphens,spaces]{url}

\usepackage{graphicx,color}
\graphicspath{{img/}}

\usepackage{tikz}
\usetikzlibrary{matrix,arrows,calc}

\usepackage{csquotes}
\usepackage[style=alphabetic,backend=bibtex,maxnames=10]{biblatex}
\bibliography{sources}

\makeatletter
\def\blx@maxline{77}
\makeatother

\usepackage{enumitem}
\setenumerate[1]{label={(\arabic*)},leftmargin=*,nolistsep,parsep=\parskip}

\usepackage{amsthm}
\usepackage{hyperref}

\makeatletter
\ifcsname phantomsection\endcsname
    \newcommand*{\@gobblenexttocentry}[9]{}
\else
    \newcommand*{\@gobblenexttocentry}[4]{}
\fi
\newcommand*{\addsubsection}{%
    \addtocontents{toc}{\protect\@gobblenexttocentry}%
    \subsection*}
\makeatother

\usepackage{aliascnt}

\newcommand{\mynewtheorem}[4]{
  \if\relax\detokenize{#3}\relax 
    \if\relax\detokenize{#4}\relax 
      \newtheorem{#1}{#2}
    \else
      \newtheorem{#1}{#2}[#4]
    \fi
  \else
    \newaliascnt{#1}{#3}
    \newtheorem{#1}[#1]{#2}
    \aliascntresetthe{#1}
  \fi
  \expandafter\def\csname #1autorefname\endcsname{#2}
}

\def\equationautorefname~#1\null{(#1)\null}

\mynewtheorem{theorem}{Theorem}{}{section}
\mynewtheorem{lemma}{Lemma}{theorem}{}
\newtheorem*{rem}{Remark}
\mynewtheorem{prop}{Proposition}{lemma}{}
\mynewtheorem{cor}{Corollary}{lemma}{}
\mynewtheorem{ex}{Example}{lemma}{}
\newtheorem*{dfn}{Definition}

\makeatletter%
\protected\def\W{\@ifnextchar[{\W@arg}{W_D}}
\def\W@arg[#1]{%
  \def\@D{D}%
  \def\c@W##1,##2,##3\@nil{\def\c@D{##1}\def\n@D{##2}}\c@W#1,,\@nil%
  \def\@W##1=##2=##3\@nil{\def\@DD{##1}\def\@DDD{##2}}\expandafter\@W\c@D==\@nil%
  \ifx\@D\@DD W_{\@DDD}%
    \expandafter\ifx\n@D\relax\relax\else(\n@D)\fi%
  \else W_D(#1)\fi%
}
\protected\def\parexp#1{\@ifnextchar^{(#1)}{#1}}
\protected\def\parind#1{\@ifnextchar_{(#1)}{#1}}

\def\defbb#1{\expandafter\def\csname b#1\endcsname{\mathbb{#1}}}
\def\defcal#1{\expandafter\def\csname c#1\endcsname{\mathcal{#1}}}
\def\deffrak#1{\expandafter\def\csname frak#1\endcsname{\mathfrak{#1}}}
\def\defop#1{\expandafter\def\csname#1\endcsname{\operatorname{#1}}}

\def\defcals#1{\@defcals#1\@nil}
\def\@defcals#1{\ifx#1\@nil\else\defcal{#1}\expandafter\@defcals\fi}
\def\deffraks#1{\@deffraks#1\@nil}
\def\@deffraks#1{\ifx#1\@nil\else\deffrak{#1}\expandafter\@deffraks\fi}
\def\defbbs#1{\@defbbs#1\@nil}
\def\@defbbs#1{\ifx#1\@nil\else\defbb{#1}\expandafter\@defbbs\fi}
\def\defops#1{\@defops#1,\@nil}
\def\@defops#1,#2\@nil{\if\relax#1\relax\else\defop{#1}\fi\if\relax#2\relax\else\expandafter\@defops#2\@nil\fi}
\makeatother%

\defbbs{ZHQCNPAR}
\defcals{AOCFGHRMXPYZEVU}
\deffraks{abcdefpqgm}
\defops{SL,mod,Spec,Re,Gal,Tr,End,GL,Hom,Kern,Bild,PGL,Aut,id,Pic,div,Supp,Fix,Jac,Im,PT}

\newcommand*{\Xfamily}{Turtle family}
\newcommand*{\Yfamily}{Hurricane family}

\def\i{\mathrm{i}}
\def\d{\mathrm{d}}

\begin{document}

\title{Orbifold points on Prym-Teichm\"{u}ller curves in genus four}

\author{David Torres-Teigell}
\address{Fachrichtung Mathematik, Universit\"{a}t des Saarlandes, Campus E24\\66123 Saarbr\"{u}cken, Germany}
\email{torres@math.uni-sb.de}

\author{Jonathan Zachhuber}
\address{FB 12 -- Institut für Mathematik\\Johann Wolfgang Goethe-Universität\\Robert-Mayer-Str. 6--8\\D-60325 Frankfurt am Main}
\email{zachhuber@math.uni-frankfurt.de}

\begin{abstract}
For each 
 discriminant $D>1$, McMullen constructed the Prym-Teichmüller curves $\W[4]$ and $\W[6]$ in $\cM_{3}$ and $\cM_{4}$, which constitute one of the few known infinite families of geometrically primitive Teichm\"{u}ller curves. 
In the present paper, we determine for each $D$ the number and type of orbifold points on $\W[6]$. These results, together with a previous result of the two authors in the genus $3$ case and with results of Lanneau-Nguyen and Möller, complete the topological characterisation of all Prym-Teichmüller curves and determine their genus.

The study of orbifold points relies on the analysis of intersections of $\W[6]$ with certain families of genus $4$ curves with extra automorphisms. As a side product of this study, we give an explicit construction of such families and describe their Prym-Torelli images, which turn out to be isomorphic to certain products of elliptic curves. We also give a geometric description of the flat surfaces associated to these families and describe the asymptotics of the genus of $W_D(6)$ for large $D$.
\end{abstract}

\maketitle

\tableofcontents

\section{Introduction}
\label{sec:intro}

A \emph{flat surface} is a pair $(X,\omega)$ where $X$ is a compact Riemann surface of genus $g$ and $\omega$ is a holomorphic differential on $X$. By integration, the differential endows $X$ with a flat structure away from the zeros of $\omega$.
Consider now $\Omega\cM_g$, the moduli space of flat surfaces which is a natural bundle over the moduli space $\cM_g$ of smooth projective curves of genus $g$. There is a natural $\SL_2(\bR)$ action on $\Omega\cM_g$ by affine shearing of the flat structure and we consider the projections of orbit closures to $\cM_g$. 
In the rare case that the $\SL_2(\bR)$ orbit of $(X,\omega)$ projects to an (algebraic) curve in $\cM_g$ we call this the \emph{Teichmüller curve generated by $(X,\omega)$} in $\cM_g$.

Not many families of (primitive) Teichmüller curves are known, see e.g.~\cite{MMW} for a brief overview. Among them, McMullen constructed the \emph{Weierstraß curves} in genus $2$ \cite{mcmtmchms} and generalised this construction to the \emph{Prym-Teichmüller curves} in genus $3$ and $4$ \cite{mcmprym}.  Recently, Eskin, McMullen, Mukamel and Wright announced the existence of six exceptional orbit closures, two of which contain an infinite collection of Teichmüller curves. One of them is treated in~\cite{MMW}.

Any Teichmüller curve $\cC$ is a sub-orbifold of $\cM_g$. Therefore, denoting by $\chi$ the (orbifold) Euler Characteristic, by $h_0$ the number of connected components, by $C$ the number of cusps and by $e_d$ the number of points of order $d$, these invariants determine the genus $g$:
\[2h_0-2g=\chi+C+\sum_d e_d\left(1-\frac{1}{d}\right),\]
i.e.\ they determine the topological type of $\cC$.

For the Prym-Weierstraß curves, the situation is as follows. In genus $2$, cusps and connected components were determined by McMullen \cite{mcmTCingenustwo}, the Euler characteristic was computed by Bainbridge \cite{bainbridgeeulerchar}, and the number and types of orbifold points were established by Mukamel \cite{mukamelorbifold}. In genus $3$ and $4$, Möller \cite{moellerprym} calculated the Euler characteristic and Lanneau and Nguyen \cite{lanneaunguyen} classified the cusps. 
The number of connected components in genus $3$ were also determined in \cite{lanneaunguyen} (see also \cite{components}) and the number and type of their orbifold points in genus $3$ were established in \cite{TTZ}. In the case of genus 4, Lanneau has recently communicated to the authors that the Prym locus is always connected~\cite{lanneaunguyen2}. The present paper classifies the orbifold points of these curves.

\begin{theorem}\label{mainthm}
For 
 discriminant $D>12$, the Prym-Teichmüller curves $W_D(6)$ have orbifold points of order $2$ and $3$.
More precisely:
\begin{itemize}
\item the number of orbifold points of order $2$ is
\[e_2(D)=\begin{cases} 0\,,&\text{if $D$ is odd,}\\h(-D)+h(-\nicefrac{D}{4})\,,&\text{if $D\equiv 12\mod 16$,}\\h(-D)\,,&\text{if $D\equiv 0,4,8\mod 16$,}\end{cases}\]
where $h(-D)$ is the class number of $\cO_{-D}$;
\item the number of orbifold points of order $3$ is
\[e_3(D)=\#\{a,i,j\in\bZ : a^2+3j^2+(2i-j)^2=D,\ \gcd(a,i,j)=1\}/12;\]
\item $W_5(6)$ has one point of order $3$ and one point of order $5$;
\item $W_8(6)$ has one point of order $2$ and one point of order $3$;
\item $W_{12}(6)$ has one point of order $2$ and one point of order $6$.
\end{itemize}
\end{theorem}

\autoref{mainthm} combines the results of \autoref{thm:e2thm}, \autoref{thm:e3thm}, and \autoref{thm:e5thm} and thus completes the topological classification of the Prym-Weierstraß curves. The topological invariants of $W_D(6)$ 
for nonsquare discriminants $D\leq 200$ are listed in \autoref{thetable} on page~\pageref{thetable}.

Recall that the orbifold locus of $W_D(6)$ consists of flat surfaces $(X,\omega)$ where $\omega$ is not only an eigenform for the real multiplication but \emph{also} for some (holomorphic) automorphism $\alpha$ of $X$. To describe this locus, it is therefore natural to consider instead families $\cF$ of curves with a suitable automorphism $\alpha$ and consider the $\alpha$-eigenspace decomposition of $\Omega\cF$. We isolate suitable eigendifferentials $\omega$ with a single zero, and check whether the Prym part of $(X,\omega)\in \Omega\cF$ admits real multiplication respecting $\omega$, i.e.\ find the intersections of $\cF$ with $W_D(6)$ for some $D$. 

\medskip 

To be more precise, it is essentially topological considerations that not only show the possible orders $d$ of orbifold points that can occur on a curve $W_D(6)$, but in fact determine the possibilities for the group $\Aut X$, in the case that $(X,\omega)$ is an orbifold point (see \autoref{sec:orbibg}).
It turns out that there are essentially two relevant families: curves admitting a $D_8$ action -- giving points of order $2$ -- and curves admitting a $C_6\times C_2$ action -- giving points of order $3$. Because of the flat picture of the single-zero differentials on these families, we will call them the \emph{\Xfamily{}} (\autoref{fig:C4unfolded}) and the \emph{\Yfamily{}} (\autoref{fig:C6unfolded}), see \autoref{sec:flat} for details. Additionally, these families intersect, giving the (unique) point of order $6$ on $W_{12}(6)$. Also, there is a unique point with a $C_{10}$ action, giving the point of order $5$ on $W_5(6)$. Any orbifold point on $W_D(6)$ must necessarily lie on one of these families (\autoref{prop:possibleorders}).

\begin{figure}[ht]
\includegraphics{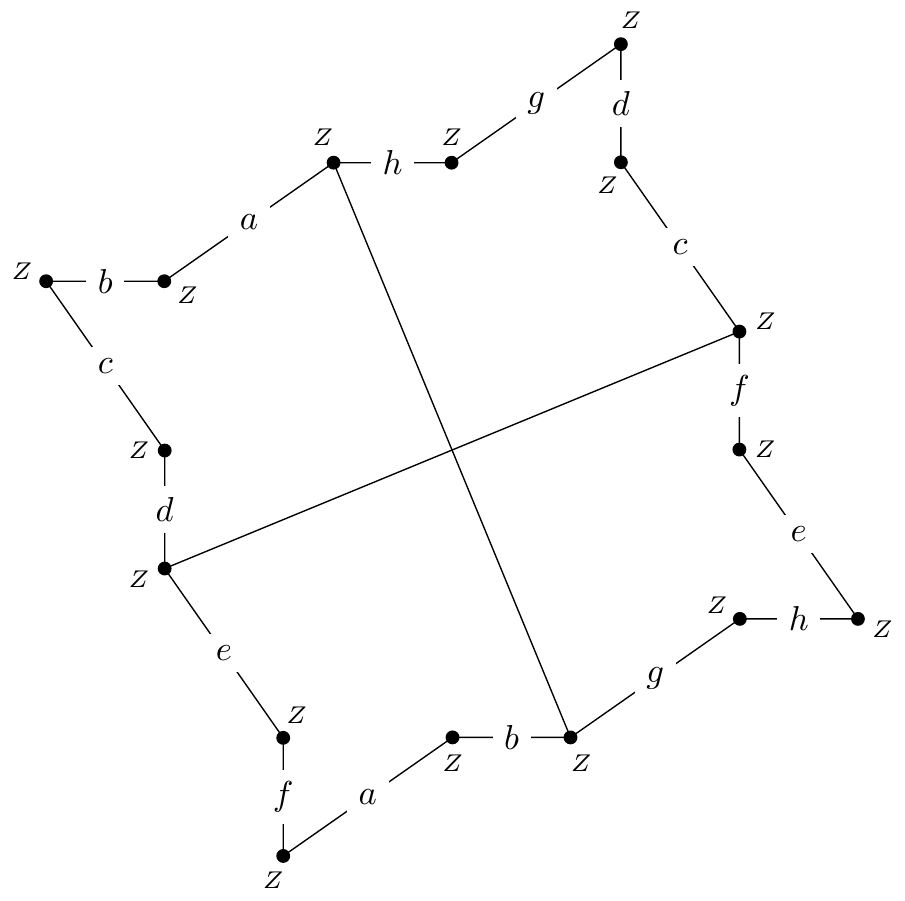}
\caption{A $C_4$-eigendifferential of genus $4$ with a single zero (the Turtle): the canonical $4$-cover of the $4$-differential on an elliptic curve pictured in \autoref{fig:C4onE} (\autoref{sec:flat}).}
\label{fig:C4unfolded}
\end{figure}

\begin{figure}[ht]
\includegraphics{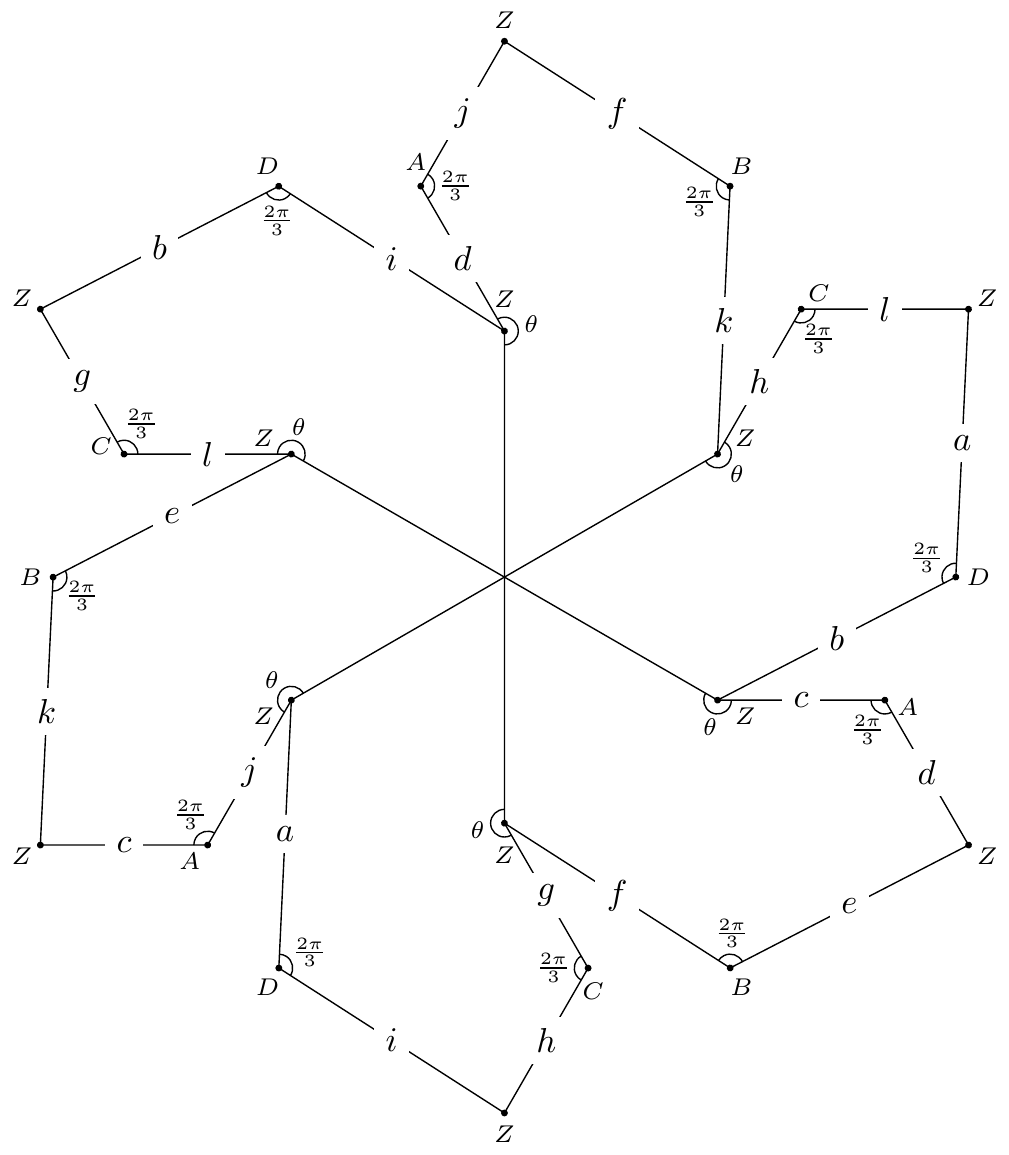}
\caption{A $C_6$-eigendifferential of genus $4$ with a single zero (the Hurricane): the canonical $6$-cover of the $6$-differential on $\bP^1$ pictured in \autoref{fig:C6onP1} (\autoref{sec:flat}).}
\label{fig:C6unfolded}
\end{figure}

The difficulty when studying these families comes from obtaining the eigenforms in a basis where we can calculate the endomorphism ring in order to study real multiplication or, equivalently, understanding the analytic representation of suitable real multiplication in the eigenbasis of the automorphism on the Prym variety.

We begin by analysing $\cM_4(D_8)$, the $2$-dimensional locus of genus $4$ curves with a specific $D_8$ action, see \autoref{sec:e2}. The \Xfamily{} is a $1$-dimensional sub-locus of this moduli space.

As a by-product, we give an explicit description of $\cM_4(D_8)$. For an elliptic curve $E$, let $\phi$ denote the elliptic involution.

\begin{theorem}\label{e2thm}
The family $\cM_4(D_8)$ is in bijection with the family
\[\cE=\{(E,[P]) : E\in\cM_{1,1},\ [P]\in(E\setminus E[2])/\phi\,\}\]
of elliptic curves with a distinguished base point, together with an 
elliptic pair.

In particular, this family is $2$-dimensional; however, the sub-locus $\cX$ of curves admitting a $C_4$-eigenform with a single zero is $1$-dimensional and in bijection with $\cM_{1,1}\backslash\{E_{2}\}$.
\end{theorem}

This bijection is induced by the construction of this family as a fibre product of two (isomorphic) families of elliptic curves over a base projective line. 

To determine which points admit real multiplication with a common eigenform, we fix an eigenbasis of $\Omega X$ and consider the \emph{Prym-Torelli map} $\PT$, which associates the corresponding Prym variety to a Prym pair $(X,\rho)$. We show that, in the $D_8$ case, the Prym variety of such a pair is always isomorphic to the product $E\times E$, where $E$ is an elliptic curve arising as a quotient of $X$, and then the Prym variety admits suitable real multiplication if and only if the elliptic curve has complex multiplication, accounting for the class numbers.

The \Yfamily{} behaves quite differently, see \autoref{sec:e3}. We denote by $E_\zeta$ the unique elliptic curve with an automorphism $\psi$ of order $6$.

\begin{theorem}\label{e3thm}
The \Yfamily{} agrees with the family 
\[\cY_{t}: y^{6} = x (x-1)^{2} (x-t)^{2}\]
of cyclic covers of $\bP^{1}$. 

However, the Prym-Torelli image of $\cY$ is the single point $E_\zeta\times E_\zeta$.
\end{theorem}

The \Yfamily{} has the advantage that it is $1$-dimensional and can be understood in terms of cyclic covers of $\bP^1$. However, due to the large automorphism group, the whole family collapses to a single point under the Prym-Torelli map, which of course admits real multiplication in many different ways. Now, each fibre $\cY_t$ gives a different $C_6$-eigenbasis in $\Omega E_\zeta\times\Omega E_\zeta$ and checking when \emph{this} basis is also an eigenbasis for some real multiplication gives the intersections of $\cY$ and some $W_D(6)$.

The \Yfamily{} can also be constructed as a family of fibre products over certain quotient curves. More precisely, all fibres $\cY_t$ of $\cY$ can be seen as a fibre product of two copies of $E_{\zeta}$ over a projective line quotient $\bP^{1}$. 
However, in contrast to the $D_8$ case, the base of the \Yfamily{} will not be isomorphic to a modular curve, but it will be a dense subset inside the curve $E_\zeta$.

More precisely, denote by $E_\zeta^*$ the curve $E_\zeta$ with the $2$-torsion points and the $\psi$-orbit of order $3$ removed and let $\phi$ be again the elliptic involution of $E_\zeta$. There is a generically 6-to-1 map between the set of elliptic pairs of points $E_\zeta^*/\phi$ and the fibres of $\cY$ (cf.~\autoref{prop:C6ellipticcurves}).

Moreover, for each isomorphism class $[Y]\in\cY$, there exist generically $12$ elements (up to scale) in $\Omega E_\zeta\times\Omega E_\zeta$ defining a $C_6$-eigendifferential with a single zero on the curves in $[Y]$ (cf.~\autoref{prop:cYfamily}). This fact explains the factor of $12$ in the formula for the number of orbifold points of order 3.

%
%

\medskip

Using the work of Möller \cite{moellerprym} and Lanneau-Nguyen \cite{lanneaunguyen}, \autoref{mainthm} lets us calculate the genus of the Prym-Weierstraß curves $W_D(6)$. In \autoref{sec:genus}, we describe the asymptotic growth rate of the genus, $g(W_D(6))$ with respect to the discriminant $D$.

\begin{theorem}\label{asymptoticsthm}
There exist constants $C_1,C_2>0$, independent of $D$, such that
\[C_1\cdot D^{3/2} < g(W_D(6)) < C_2\cdot D^{3/2}.\]
Moreover, $g(\W)=0$ if and only if $D\leq 20$.
\end{theorem}

The topological invariants of the geometrically primitive genus $0$ Prym-Teichmüller curves are summarized in \autoref{smalltable}.

\begin{table}
\caption{Topological invariants of the six Teichmüller curves $W_D(6)$ that have genus $0$. The number of cusps is described in \cite{lanneaunguyen}, the Euler characteristic in \cite{moellerprym}. For larger $D$, see \autoref{thetable} on page~\pageref{thetable}.}
\label{smalltable}
\begin{tabular}{r  r  r  r  r  r r r}\toprule
  $D$ & $g$ & $e_2$ & $e_3$ & $e_5$ & $e_6$ & $C$ & $\chi$  \\\midrule
$5$&$0$ & $0$ & $1$ & $1$& $0$ & $1$ &$-\nicefrac{7}{15}$\\
$8$&$0$ & $1$ & $1$ & $0$ & $0$ & $2$ &$-\nicefrac{7}{6}$\\
$12$&$0$&$1$ & $0$ & $0$ & $1$ & $3$ &$-\nicefrac{7}{3}$\\
$13$&$0$&$0$ & $2$ & $0$ & $0$ & $3$ &$-\nicefrac{7}{3}$\\
$17$&$0$&$0$ & $1$ & $0$ & $0$ & $6$ &$-\nicefrac{14}{3}$\\
$20$&$0$&$2$ & $1$ & $0$ & $0$ & $5$ &$-\nicefrac{14}{3}$\\
\bottomrule
\end{tabular}
\end{table}

\medskip

\autoref{mainthm} can be seen as the next and final step after \cite{mukamelorbifold} and \cite{TTZ} in the study of orbifold points on Prym-Weierstraß curves, thus bringing closure to the topological characterization of such curves. While the general method is similar in the genus $2$, $3$ and $4$ cases (namely, studying the intersection of the Teichm\"{u}ller curves with certain families), the specific phenomena occurring are different.

In genus $2$, the situation was simpler essentially due to the fact that the Prym variety was the entire Jacobian \cite{mukamelorbifold}. While the relevant family also had a generic $D_8$ automorphism group, this was a $1$-dimensional object, while in $\cM_4$ this locus is a surface where the $C_4$-eigendifferentials are contained in an embedded Modular curve.

In genus $3$, the defining phenomenon was the fact that the Prym variety was a $(2,1)$-polarised abelian sub-variety of the Jacobian \cite{TTZ}. Also, for the first time, two $1$-dimensional families occured and in the case of $C_4$ curves, the Prym-Torelli image also collapsed to a point. However, both these families could be described as cyclic covers of $\bP^1$ in which case the eigenspace decomposition of $\Omega X$ is well understood. The main technical difficulty in those cases was the explicit calculation of period matrices using Bolza's method.
Moreover, the formulas obtained were of a slightly different flavour, as  the $C_6$-family turned out to be isomorphic to the \emph{compact} Shimura curve $\bH/\Delta(2,6,6)$, giving a more general class number than in the other cases.

In contrast, in genus $4$, for the first time a $2$-dimensional locus plays a central role:
indeed, the space $\cM_4(D_8)$ can be seen as a cyclic cover over an elliptic curve, which makes the computation of the eigendifferentials with a single zero more difficult, cf. \autoref{thm:cXfamily}.
On the other hand, while in almost all cases the Prym variety is isogenous to a product of elliptic curves (this is the reason for the abundance of modular curves and class numbers in the formulas), it turns out that in genus $4$ the Prym varieties are actually \emph{isomorphic} to this product. This results in a closer relationship of the endomorphism rings in this case and is the reason we obtain an exact class number of a negative discriminant order in \autoref{mainthm}.

In particular, the technical approach in this paper is completely different than the one in~\cite{TTZ}, since the computational aspects of Bolza's method have been replaced by a more conceptual description of the families.

The occuring positive dimensional families are summarised and compared to the families occuring in genus $2$ and $3$ in \autoref{families}.

\begin{table}
\caption{Families parametrising possible orbifold points on Prym-Weierstraß curves. Here, $\cX_t$ is a general fibre of $\cX$ and $E_t$ is an elliptic curve that appears as a quotient of $\cX_t$. For genus $2$ see \cite{mukamelorbifold}, for genus $3$ see \cite{TTZ}.}
\label{families}
\begin{tabular}{c c c c c c}
\toprule
$g(\cX_t)$&$d$&$\dim\cX$&$\dim\PT(\cX)$&$\Aut(\cX_t)$&$\End(\cP(\cX_t,\sigma^d))$\\
\midrule
$2$&$2$&$1$&$1$&$D_8$&order in $M_2(\End(E_t))$\\
$3$&$2$&$1$&$0$&$C_2\ltimes(C_2\times C_4)$&order in $M_2(\bQ[i])$\\
$3$&$3$&$1$&$1$&$C_6$&order in $\bigl(\frac{2,-3}{\bQ}\bigr)$\\
$4$&$2$&$2$ $(1)$&$1$&$D_8$&$M_2(\End(E_t))$\\
$4$&$3$&$1$&$0$&$C_6\times C_2$&$M_2(\bZ[\zeta_6])$\\
\bottomrule
\end{tabular}
\end{table}

Finally, in \autoref{sec:flat}, we provide the flat pictures associated to the eigendifferentials in the \Xfamily{} and the \Yfamily.

\subsection*{Acknowledgements}
We are very grateful to Martin Möller for many useful discussions and his constant encouragement to also complete the genus $4$ case.
We also thank \cite{PARI2} for computational assistance.

\section{The orbifold locus of \texorpdfstring{$\W$}{W\textunderscore{}D}} 
\label{sec:orbibg}

The aim of this section is to describe the orbifold locus of $W_D(6)$ as the intersection with families of curves with a prescribed automorphism group in $\cM_4$. In particular, \autoref{prop:possibleorders} determines the possible orders of orbifold points that may occur.

As usual, we write $( g ; n_{1},\ldots,n_{r})$ for orbifolds of genus $g$ with $r$ points of order $n_{1},\ldots,n_{r}$. Recall that, given an automorphism $\alpha$ of order $N$ on $X$, points of order $n_{i}$ on $X/\alpha$ correspond to orbits of length $N/n_{i}$ on $X$. Moreover, we denote by $\zeta_d$ a primitive $d$th root of unity.

\begin{prop}\label{prop:possibleorders}
Let $(X,\omega)\in\Omega\W$ be a flat surface that parametrises an orbifold point of $\W$ of order $d$. Then there exists a holomorphic automorphism $\alpha\in\Aut X$ of order $2d$ that satisfies $\alpha^*\omega=\zeta_{2d}\omega$ and one of the following conditions:
\begin{enumerate}
	\item the order of $\alpha$ is $4$ and $X/\alpha$ has signature $( 1 ; 4,4 )$;
	\item the order of $\alpha$ is $6$ and $X/\alpha$ has signature $( 0 ; 3,3,6,6 )$;
	\item the order of $\alpha$ is $10$ and $X/\alpha$ has signature $( 0 ; 5,10,10 )$;
	\item the order of $\alpha$ is $12$ and $X/\alpha$ has signature $( 0 ; 3,12,12 )$.
\end{enumerate}
\end{prop}

\begin{rem}
Observe that family (1) is $2$-dimensional, family (2) is $1$-dimensional and families (3) and (4) consist of a finite number of points in $\cM_4$.
\end{rem}

Before we proceed to the proof, we first recall some background and notation.

\subsection*{Flat Surfaces and Teichmüller Curves}

A \emph{flat surface} is a pair $(X,\omega)$ where $X$ is a Riemann surface (or equivalently a smooth irreducible complex curve) of genus $g$ and $\omega\in\Omega X$ is a holomorphic differential form on $X$. Note that $X$ may be endowed with a flat structure away from the zeros of $\omega$ via integration. We will denote the moduli space of flat surfaces by $\Omega\cM_g$ and note that it can be viewed as a bundle $\Omega\cM_g\to\cM_g$ over the moduli space of smooth, irreducible, complex curves of genus $g$. The space $\Omega\cM_g$ is naturally stratified by the distribution of the zeros of the differentials; given a partition $\mu=(\mu_{1},\ldots,\mu_{r})$ of $2g-2$, denote by $\Omega\cM_g(\mu)$ the corresponding stratum and given a family $\cF$ of curves in $\cM_g$ we set $\Omega\cF(\mu)\coloneqq\Omega\cM_g(\mu)\cap\Omega\cF$. We will use exponential notation for repeated indices, so that, for instance $(1,\ldots,1)=(1^{2g-2})$.

Recall that $\Omega\cM_g$ admits a natural $\GL_2(\bR)$ action by affine shearing of the flat structures. A \emph{Teichmüller curve} is the (projection of a) $\GL_2(\bR)^+$ orbit that projects to an algebraic curve in $\cM_g$. See for instance~\cite{tmcintro} for background on Teichmüller curves and flat surfaces.

\subsection*{Prym-Teichmüller Curves in Genus \texorpdfstring{$4$}{4}}

McMullen \cite{mcmprym} constructed families of primitive Teichmüller curves in genus $2$, $3$ and $4$, the \emph{Prym-Teichmüller} (or \emph{Prym-Weierstraß}) curves $\W[2g-2]$. We briefly recall the construction in the genus $4$ case. For brevity, we denote the curve $W_D(6)$ by $\W$ in the following. 

Let $X$ be of genus $4$ admitting a holomorphic involution $\rho$. We say that $\rho$ is a \emph{Prym involution} if $X/\rho$ has genus $2$. In particular, this gives a decomposition $\Omega X=\Omega X^+\oplus\Omega X^-$ into $2$-dimensional $\rho$-eigenspaces with eigenvalues $1$ and $-1$ respectively. It also determines sublattices $H_{1}(X,\bZ)^{+},H_{1}(X,\bZ)^{-}\subset H_{1}(X,\bZ)$ consisting of $\rho$-invariant and $\rho$-anti-invariant cycles that satisfy $H_{1}(X,\bZ)^{\pm}=H_{1}(X,\bZ)\cap (\Omega X^{\pm})^{*}$. All this implies that the \emph{Prym variety}
\[\cP(X,\rho) \coloneqq \frac{(\Omega X^{-})^{*}}{H_{1}(X,\bZ)^{-}}=\ker(\Jac(X)\to\Jac(X/\rho))^0\]
is a $2$-dimensional, $(2,2)$-polarised abelian sub-variety of the Jacobian $\Jac(X)$ (see~\cite{moellerprym} or~\cite[Ch. 12]{birkenhakelange} for details).

For any discriminant $D\equiv 0,1\mod 4$, write $D=b^2-4ac$ for some $a,b,c \in\bZ$. The (unique) quadratic order of discriminant $D$ is defined as $\cO_D=\bZ[T]/(aT^2+bT+c)$, which agrees with
\[\cO_D=\bZ\oplus T_D\bZ,\text{ where } 
T_D=\begin{cases}\frac{\sqrt{D}}{2},&\text{if $D\equiv 0\mod 4$ and}\\\frac{\sqrt{D}+1}{2},&\text{if $D\equiv 1\mod 4$,}\end{cases}
\]
provided $D$ is not a square. If $D=f^{2}$, the order $\cO_{D}=\bZ[T]/(T^2-fT)$ is isomorphic to the subring $\{(a,b)\in\bZ\oplus\bZ\,:\, a\equiv b\mod{f}\}$.

Now let $D>0$ be a positive discriminant. We say that a polarised abelian surface $A$ has \emph{real multiplication} by $\cO_D$ if it admits an embedding $\cO_D\hookrightarrow\End A$ that is self-adjoint with respect to the polarisation. We call the real multiplication by $\cO_D$ \emph{proper}, if the embedding cannot be extended to any quadratic order containing $\cO_D$.


Denote by $\Omega\W$ the locus of $(X,\omega)\in\Omega\cM_4(6)$ such that
\begin{enumerate}
\item $X$ admits a Prym involution $\rho$, so that $\cP(X,\rho)$ is $2$-dimensional,
\item the form $\omega$ has a single zero and satisfies $\rho^*\omega=-\omega$, and
\item $\cP(X,\rho)$ admits proper real multiplication by $\cO_D$ with $\omega$ as an eigenform.
\end{enumerate}
McMullen showed that the projection $\W$ of $\Omega\W$ to $\cM_4$ gives (a union of) Teichmüller curves for every discriminant $D$ \cite{mcmprym}. In fact, Lanneau has communicated to the authors that $\W$ is connected for all $D$ \cite{lanneaunguyen2}.

\subsection*{Orbifold Points on Prym-Teichmüller Curves}

An orbifold point of order $d$ on $\W$ corresponds to a flat surface $(X,\omega)\in\Omega\W$  such that
\begin{itemize}
\item there exists a holomorphic automorphism $\alpha\in\Aut X$, such that $\alpha^*\omega=\lambda\,\omega$ for some $\lambda\in\bC^*\setminus\{\pm 1\}$;
\item the element $\rho=\alpha^{d}$ is a Prym involution satisfying $\rho^*\omega=-\omega$;
\item $\omega$ is an eigenform for real multiplication on the Prym variety $\cP(X,\rho)$.
\end{itemize}
Note that this implies that $\alpha$ is of order $2d$ and must have a fixed point (at the single zero of $\omega$). Details and background can be found in \cite{TTZ}.

\begin{dfn}
We will say that $(\omega_1,\omega_2)$ is an \emph{$\alpha$-eigenbasis} of $\Omega X^-$ if the $\omega_i$ are both eigenforms for the action of $\alpha^{*}$.
\end{dfn}

To study these points, we study the locus of curves in $\cM_4$ with an appropriate automorphism $\alpha$ and an eigenform with a single zero.

\begin{proof}[Proof of \autoref{prop:possibleorders}]
Let $(X,\omega)$ correspond to an orbifold point in $\W$ of order $d$. The Prym involution $\rho$ in genus $4$ gives a genus $2$ quotient with two fixed points, i.e. $X/\rho \cong ( 2 ; 2^{2} )$. By the argument above, the curve $X$ must possess an automorphism $\alpha$ of order $2d$ that admits $\omega$ as an eigenform with eigenvalue $\zeta_{2d}$, has at least one fixed point and satisfies $\alpha^d = \rho$. 

The automorphism $\alpha$ descends to an automorphism of $X/\rho$ of order $d$ and, looking at possible orders of automorphisms on curves of genus $2$, one sees that $d=2,3,4,5,6,8$ or $10$ (see for instance~\cite{Sch,Bro}). Now, points of odd order $k$ on $X/\alpha$ (equivalently, $\alpha$-orbits of length $2d/k$ on $X$) give unramified points on $X/\rho$, since they are not fixed by $\rho=\alpha^d$ (more precisely, their preimages on $X$ are not fixed). Points of even order $2k$ on $X/\alpha$ (equivalently, $\alpha$-orbits of length $d/k$ on $X$) give $d/k$ ramified points on $X/\rho$.

Since there are only two ramification points on $X/\rho$ and at least one of them necessarily comes from a fixed point of $\alpha$, the automorphism $\alpha$ has two fixed points and no more ramification points of even order. A case-by-case analysis using Riemann-Hurwitz yields the four options in the statement.
%
%
%
\end{proof}

%
%
%

\subsection*{Products of elliptic curves}

In the analysis of orbifold points on Prym-Teichmüller curves in genus $4$, Klein-four actions and products of elliptic curves will be omnipresent. The following result will be a crucial technical tool.

\begin{prop}\label{prop:PrymIsomorphism} Let $(X,\rho)$ be a genus $4$ curve in the Prym locus admitting a Klein four-group of automorphisms $V_{4}=\langle \rho,\beta\rangle$ such that both $X/\beta$ and $X/\rho\beta$ have genus $1$. Then $\cP(X,\rho)\cong X/\beta \times X/\rho\beta$ as $(2,2)$-polarised abelian varieties.
\end{prop}

\begin{rem}
Note that this is in stark contrast to the situation in genus $2$ and $3$. Although in those cases $V_4$ actions were also ubiquitous, the Prym variety was always only \emph{isogenous} to a product of elliptic curves (cf. \cite[Proposition 2.13]{mukamelorbifold} and \cite[Theorem 1.2]{TTZ}) as the Prym variety is $(1,1)$, respectively $(2,1)$, polarised in those situations. In the genus $4$ case, the result above yields an even stronger relationship between the geometry of the quotient elliptic curves and the Prym variety.
\end{rem}


Let us first recall some general facts about elliptic curves. An elliptic curve $E\coloneqq (E,O)\in\cM_{1,1}$ is a smooth genus $1$ curve together with a chosen base point $O\in E$. It always admits the structure of a group variety with neutral element $O$. The set of $2$-torsion points with respect to this group law consists of four elements and is usually denoted by $E[2]$.
 
Every elliptic curve is isomorphic to $E_{\lambda}\coloneqq\{v^{2}=u(u-1)(u-\lambda)\}$ for some $\lambda\in\bC\backslash\{0,1\}$, where we choose the base point $O$ to be the point at infinity. Permuting $\{0,1,\infty\}$ gives an isomorphism between the elliptic curves corresponding to
	\[\lambda\,,\ 1-\lambda\,,\ \frac{1}{\lambda}\,,\ \frac{1}{1-\lambda}\,,\ \frac{\lambda-1}{\lambda}\,,\ \frac{\lambda}{\lambda-1}\,.\]

By the uniformisation theorem, every elliptic curve can also be represented as the quotient of $\bC$ by a lattice $\Lambda=\bZ \oplus \tau\bZ$ for some $\tau$ in the upper half-plane $\bH$. Points in the same $\SL_{2}\bZ$-orbit yield isomorphic elliptic curves, and therefore one can realise the moduli space of elliptic curves $\cM_{1,1}$ as the quotient $\bH/\SL_{2}\bZ$. The relationship between $\tau$ and $\lambda$ is given by the modular $\lambda$-function.

Each elliptic curve carries a natural \emph{elliptic involution} $\phi$, the set of fixed points of which agrees with $E[2]=\Fix(\phi)$. In the model $E_{\lambda}$, the elliptic involution is given by $(u,v)\mapsto (u,-v)$ and one has $E_{\lambda}[2]=\{(0,0),(1,0),(\lambda,0),\infty\}$. The quotient by the elliptic involution is isomorphic to $\bP^{1}$.

The general element of $\cM_{1,1}$ has no further automorphisms fixing the base point. The only exceptions, which correspond to the orbifold points of $\bH/\SL_{2}\bZ$, are $E_{2}$ (corresponding to $\tau=\i$ in the upper half-plane) with a cyclic automorphism group of order $4$, and $E_{\zeta_{6}}$ (corresponding to $\tau=\zeta_6$ in the upper half-plane), where $\zeta_{6}=e^{2\pi\i/6}$, with a cyclic automorphism group of order $6$.

\begin{proof}[Proof of~\autoref{prop:PrymIsomorphism}]Consider the quotients $p_{1}\colon X\to X/\beta$ and $p_{2}\colon X\to X/\rho\beta$. Since $X/V_{4}$ has genus $0$, the images of the pullbacks $p_{1}^{*}\colon\Omega(X/\beta)\to \Omega X$ and $p_{2}^{*}\colon\Omega(X/\rho\beta)\to \Omega X$ must both lie in $\Omega X^{-}$, the $-1$-eigenspace of $\rho$ and in fact generate $\Omega X^{-}$. Therefore, denoting also by $p_{i}^{*}$ the induced map between Jacobians and identifying the elliptic curves with their Jacobians, one has 
\[\begin{tikzpicture}
\matrix (m) [matrix of math nodes, row sep=1.75em, column sep=3.5em, text height=1.5ex, text depth=0.25ex]
{ \Omega(X/\beta)\times \Omega(X/\rho\beta) & \Omega X^{-} & \Omega X \\
X/\beta\times X/\rho\beta & \cP(X,\rho) & \Jac X \\
};
\path[->,font=\scriptsize]
(m-1-1) edge [above] node {$\cong$} (m-1-2)
(m-2-1) edge [above] node {$(p_{1}^{*},p_{2}^{*})$} (m-2-2)
(m-1-1) edge [above] node {} (m-2-1)
(m-1-2) edge [above] node {} (m-2-2)
(m-1-3) edge [above] node {} (m-2-3)
;
\path[right hook->,font=\scriptsize]
(m-1-2) edge [above] node {} (m-1-3)
(m-2-2) edge [above] node {} (m-2-3)
;
\end{tikzpicture}\]
Since the polarisations induced from $\Jac X$ on $X/\beta\times X/\rho\beta$ and on $\cP(X,\rho)$ are both of type $(2,2)$, the map $(p_{1}^{*},p_{2}^{*})$ is necessarily an isomorphism of polarised abelian varieties.
\end{proof}

In particular, the proof shows that, in the above situation, we have a natural decomposition
\[\Omega X^{-}=\Omega X_{\beta}^{+}\oplus\Omega X_{\rho\beta}^{+}\]
into a $\beta$ and $\rho\beta$ invariant subspace consisting of the differential forms that arise as pullbacks from the two quotient elliptic curves.

\begin{dfn}\label{productbasis}
Let $X$ be a genus $4$ curve with a $V_4$ action.
We say that $(\eta_1,\eta_2)$ is a \emph{product basis} of $\Omega X^{-}$ if $\eta_1\in\Omega X_{\beta}^{+}$ and $\eta_2\in\Omega X_{\rho\beta}^{+}$.
\end{dfn}

Note that any product basis is a $\beta$-eigenbasis. More precisely, we have
\begin{equation}\label{betabasis}
\beta^*\eta_1=\eta_1\quad\text{and}\quad\beta^*\eta_2=-\eta_2,
\end{equation}
as $\eta_2$ is $\rho$-anti-invariant.

\section{Points of Order \texorpdfstring{$2$}{2} and  \texorpdfstring{$6$}{6}}\label{sec:e2}

The aim of this section is to prove the following formula describing the number of points of order $2$ on each Teichmüller curve $\W$ and the unique point of order $6$ on $\W[D=12]$ (cf. \autoref{mainthm}). Let $h(-C)$ denote the class number of the imaginary quadratic order $\cO_{-C}$.

\begin{theorem}\label{thm:e2thm}
Let $D\neq 12$ be a 
 positive discriminant. 
\begin{itemize}
\item If $D\equiv 1\mod 4$ then $\W$ has no orbifold points of order $2$.
\item If $D\equiv 12\mod 16$ then $\W$ has $h(-D)+h(-\frac{D}{4})$ orbifold points of order $2$.
\item Otherwise, $\W$ has $h(-D)$ orbifold points of order $2$.
\end{itemize}
Moreover, $\W[D=12]$ has one point of order $2$ and one point of order $6$.
\end{theorem}

To prove this theorem, we begin by a careful analysis of genus $4$ curves admitting an automorphism of order $4$ with two fixed points. 

\subsection*{Curves admitting an automorphism of order 4}
By \autoref{prop:possibleorders}, for $(X,\omega)$ to parametrise a point of order $2$ on $\W$, the curve $X$ must necessarily lie in the locus of curves with an automorphism of order $4$ with two fixed points. In fact, all such curves admit a faithful $D_{8}$ action.

\begin{lemma}\label{prop:D8crit}
Let $X$ be a curve of genus $4$ and $\alpha\in\Aut X$ an automorphism of order $4$ with two fixed points. 

Then $X/\alpha^2$ is of genus $2$ and there exists an involution $\beta\in\Aut X$ such that $\alpha\beta=\beta\alpha^{-1}$, i.e. $\langle\alpha,\beta\rangle\leq\Aut X$ is a $D_8$.
\end{lemma}

\begin{proof}
Since the quotient $X/\alpha$ by such an automorphism yields a curve of genus $1$ with two orbifold points of order $4$, this is just case~$N2$ in~\cite{BuCo}. The proof of Bujalance and Conder relies on a previous result by Singerman~\cite[Thm.~1]{Sin} stating that every Fuchsian group with signature $(1 ; t,t)$ is included in a Fuchsian group with signature $(0;2^{4},t)$. This group corresponds to the quotient $X/\langle\alpha,\beta\rangle$.
\end{proof}

In terms of the corresponding curves, the situation is the following. The automorphism $\alpha$ descends to an involution $\overline{\alpha}$ of the genus $2$ curve $X/\alpha^{2}$ different from the hyperelliptic involution $\overline{\beta}$. The hyperelliptic involution lifts to an involution $\beta$ on $X$ which, together with $\alpha$, generates the dihedral group. We will denote by $p_{1}:X\to X/\beta$ and $p_{2}:X\to X/\alpha^{2}\beta$ the corresponding projections.

\begin{dfn}We will denote by $\cM_{4}(D_{8})$ the family of genus $4$ curves admitting an automorphism of order $4$ with two fixed points. 
\end{dfn}

\begin{rem}
Note that by~\autoref{prop:D8crit}, this family agrees with the moduli space of Riemann surfaces of genus $4$ with $D_{8}$-symmetry, where we fix the topological action as in the lemma. Moreover, moduli spaces of curves have been studied intensively, see e.g. \cite{GoHa} for background and notations.
\end{rem}

It turns out that such a curve is essentially determined by its genus $1$ quotients. 

\begin{prop}\label{prop:D8ellipticcurves}
The family $\cM_{4}(D_{8})$ is in bijection with the family
	\[\cE=\left\{ (E,[P])\,:\, E\in\cM_{1,1},\ [P]\in (E\backslash E[2])/\{\pm1\}\,\right\}\]
of elliptic curves with a distinguished base point, together with 
an elliptic pair.

The bijection is given by $X \mapsto (X/\beta,\,[p_{1}(\Fix(\alpha\beta))])$, where the origin of the elliptic curve is chosen to be the point $p_{1}(\Fix(\alpha))$.
\end{prop}

\begin{rem}
Note that this is a $2$-dimensional locus inside $\cM_4$. However, we will show in~\autoref{thm:cXfamily} that the sub-locus $\cX$ where $X$ admits an $\alpha$-eigenform in $\Omega X^{-}$ with a single zero is in fact $1$-dimensional.
\end{rem}

This classification is obtained by a careful analysis of the ramification data of $D_{8}\cong\langle\alpha,\beta\rangle\le\Aut X$.

Consider the following diagram of ramified covers:
\[\begin{tikzpicture}
\matrix (m) [matrix of math nodes, row sep=1.75em, column sep=3.5em, text height=1.5ex, text depth=0.25ex]
{ & X & \\
X/\alpha^{2} & X/\beta & X/\alpha^{2}\beta \\
X/\alpha & X/\langle\alpha^{2},\beta\rangle & \\
 & X/\langle\alpha,\beta\rangle & \\
};
\path[->,font=\scriptsize]
(m-1-2) edge [above] node {$p$} (m-2-1)
(m-1-2) edge [right] node {$p_{1}$} (m-2-2)
(m-1-2) edge [above] node {$p_{2}$} (m-2-3)
(m-2-1) edge (m-3-2)
(m-2-1) edge (m-3-1)
(m-2-2) edge [right] node {$\pi_{1}$} (m-3-2)
(m-2-3) edge [below] node {$\pi_{2}$} (m-3-2)
(m-3-2) edge (m-4-2)
(m-3-1) edge (m-4-2)
;
\end{tikzpicture}\]
Observe that all maps in the diagram are of degree $2$.

The involutions $\beta$ and $\alpha^{2}\beta$ each have $6$ fixed points on $X$. Together they form three orbits of length $4$ under $\langle\alpha,\beta\rangle$. Similarly, $\alpha\beta$ and $\alpha^{-1}\beta$ have $2$ fixed points each, forming a whole orbit under $\langle\alpha,\beta\rangle$ of length $4$. Now, the four points of order $2$ in $X/\langle\alpha,\beta\rangle$ correspond to the (three) orbits of the fixed points of $\beta$ and $\alpha^{2}\beta$ plus the orbit of the fixed points of $\alpha\beta$ and $\alpha^{-1}\beta$.

Looking at the ramification data of $\alpha$ and $\beta$, one sees that the quotients $X/\beta$ and $X/\alpha^{2}\beta$ by $\beta$ and $\alpha^{2}\beta = \alpha\beta\alpha^{-1}$ respectively correspond to curves of genus $1$.
Choosing the image of $\Fix(\alpha)$ as an origin on each quotient, they are in fact isomorphic as elliptic curves, since $\beta$ and $\alpha^{2}\beta$ are conjugate.


Also, the above-described action of $\alpha^2\beta$ and $\beta$ may be described purely in terms of the quotient maps:
the six branch points of $p_1$ are mapped via $p_2$ to the three $2$-torsion points on $X/\alpha^2\beta$, while $p_1$ maps the six branch points of $p_2$ to the three $2$-torsion points on $X/\beta$.

\begin{proof}[Proof of \autoref{prop:D8ellipticcurves}]
Denote by $\phi$ the elliptic involution on $E$ and let $\varphi\colon E\to\bP^{1}$ be the corresponding quotient map, which we normalise such that $\varphi(O)=\infty$ and $\varphi(P)=0$. We define $X=X_{(E,[P])}$ as the fibre product of the diagram 
\[E\xrightarrow{\ \varphi\,} \mathbb{P}^{1} \xleftarrow{-\varphi}E.\]
Note that, although there is a degree of freedom in choosing $\varphi$, this does not affect the construction.

It is obvious from the ramification data that $X_{(E,[P])}$ has genus 4, and the automorphisms $(Q_{1},Q_{2})\mapsto(Q_{2},\phi(Q_{1}))$ and $(Q_{1},Q_{2})\mapsto(Q_{1},\phi(Q_{2}))$ of $E\times E$ restrict to automorphisms $\alpha$ and $\beta$ of $X_{(E,[P])}$ generating a $D_{8}$. It is straightforward to check that the map $(E,[P])\mapsto X_{(E,[P])}$ thus defined is inverse to $X \mapsto (X/\beta,\,[p_{1}(\Fix(\alpha\beta))])$.
\end{proof}


In particular, these curves satisfy the assumptions of~\autoref{prop:PrymIsomorphism}, and therefore their Prym varieties are isomorphic to a product of elliptic curves.

\begin{cor}\label{cor:D8Prym}Let $X\in\cM_{4}(D_{8})$.
Then $\cP(X,\alpha^2)\cong E\times E$ as polarised abelian varieties, where $E\cong X/\beta\cong X/\alpha^{2}\beta$.
\end{cor}

As the quotient elliptic curves are isomorphic, we pick some differential form $\eta_E$ on $E$ and denote by 
\begin{equation}\label{D8pullbackdiffs}
\eta_i=p_i^*\eta_E,\quad\text{for $i=1,2$,}
\end{equation} the corresponding product basis. Using the explicit description of $X$ as a fibre product and the expression of $\alpha$ used in the proof of~\autoref{prop:D8ellipticcurves}, one can easily describe the action of $\alpha^{*}$ on these differentials to see that
	\[\alpha^*\eta_1=-\eta_2\quad\mbox{and}\quad\alpha^*\eta_2=\eta_1\,.\]
In particular, $\alpha$ interchanges the spaces $\Omega X_{\beta}^{+}$ and $\Omega X_{\rho\beta}^{+}$.

%

\subsection*{The Eigenspace Decomposition}
For a $D_8$ curve to parametrise an orbifold point, it must necessarily admit an $\alpha$-eigenform with a single ($6$-fold) zero. To determine the possible eigenforms, we must analyse the decomposition of $\Omega X$ into $\alpha$-eigenspaces.
We denote, as usual, by $\Omega X^{-}$ and $\Omega X^{+}$ the $-1$- and $+1$-eigenspaces of $\Omega X$ with respect to the (Prym) involution $\alpha^{2}$.

\begin{prop}\label{prop:D8differentials}
Let $X\in\cM_{4}(D_{8})$. There is a natural splitting
\[\Omega X^{-}=\Omega X_{\alpha}^{\i}\oplus\Omega X_{\alpha}^{-\i}\]
into $\pm\i$-eigenspaces of $\alpha$. The spaces $\Omega X_{\alpha}^{\pm\i}$ are interchanged by $\beta$.

%
\end{prop}

\begin{proof}
The quotient $X/\alpha$ has genus $1$, so it is obvious that $\Omega X^{+}$ decomposes into $\alpha$-eigenspaces of dimension $1$ with eigenvalue $+1$ and $-1$. 

On the other hand, since $\alpha\beta=\beta\alpha^{-1}$, if $\alpha^{*}\omega=\lambda\omega$ for some $\lambda\in\bC$, clearly $\alpha^{*}(\beta^{*}\omega)=\lambda^{-1}\beta^{*}\omega$. In particular, the eigenvalues of $\alpha^{*}$ on $\Omega X^{-}$ can only be $\pm\i$, therefore the space necessarily decomposes as the sum of the $\alpha^{*}$ $\i$-eigenspace and the $-\i$-eigenspace. 
\end{proof}
%

Note that any $\alpha$-eigenbasis $(\omega_1,\omega_2)$ of $\Omega X^-$ will satisfy $\omega_1\in\Omega X_{\alpha}^{\i}$ and $\omega_2\in\Omega X_{\alpha}^{-\i}$, up to renumbering. Moreover, any product basis $(\eta_1,\eta_2)$ as in \autoref{D8pullbackdiffs} gives rise to an $\alpha$-eigenbasis
\begin{equation}\label{eq:eigencoordinates}
\omega_1=\eta_1+\i\eta_2,\quad\omega_2=\eta_1-\i\eta_2\,.
\end{equation}


\

Now, while the family $\cM_4(D_{8})$ of curves admitting a $D_8$ action is $2$-dimensional, it turns out that requiring an $\alpha$-eigenform with a single zero reduces the dimension of the locus we are interested in by one. Let us define
	\[\cX=\left\{X\in\cM_{4}(D_{8}) : \exists\ \omega\in\Omega X^{-} \mbox{ $\alpha$-eigenform with a single zero}\, \right\}\,.\]

Because of the flat picture of the elements $(X,\omega)$ in $\Omega\cX^{-}(6)$, we will call $\cX$ the \emph{\Xfamily{}} (see~\autoref{sec:flat}).

\begin{theorem}\label{thm:cXfamily}The map
	\[\begin{array}{lll}
	\cX & \longrightarrow & \cM_{1,1} \\
	X & \longmapsto & X/\beta\,,
	\end{array}\]
where the origin of $X/\beta$ is chosen to be $p_{1}(\Fix(\alpha))$, induces a bijection between $\cX$ and $\cM_{1,1}\backslash\{E_{2}\}$.

The only curve in $\cX$ where $\alpha$ is extended by an automorphism of order $12$ is the one corresponding to $X/\beta\cong E_{\zeta_{6}}$. It agrees with family $(4)$ in~\autoref{prop:possibleorders}.
\end{theorem}

\begin{proof}
By~\autoref{eq:eigencoordinates} $\alpha$-eigenforms in $\Omega X^{+}$ are given, up to scale, by $\omega_{1}=p_{1}^{*}\eta_{E}+\i p_{2}^{*}\eta_{E}$ and $\omega_{2}=p_{1}^{*}\eta_{E}-\i p_{2}^{*}\eta_{E}$. We will proceed in several steps.

\medskip

{\it Step 1:} The $\alpha$-eigenforms in $\Omega X^{-}$ can have a zero at most at one of the (two) fixed points of $\alpha$. 

Otherwise, every differential in $\Omega X^{-}$ would vanish at both fixed points of $\alpha$. In particular, so would $p_{1}^{*}\eta_{E}$ and $p_{2}^{*}\eta_{E}$. But the maps $p_i$ are unramified at $\Fix(\alpha)$ and we know that $\eta_{E}$ has no zeroes in $E$. Note that, since zeroes of $\alpha$-eigenforms outside $\Fix(\alpha)$ must be permuted by $\alpha$, this immediately implies that the differentials $\omega_{1}$ and $\omega_{2}$ lie either in $\Omega\cX(1^{4},2)$ or in $\Omega\cX(6)$. Hence it remains to show that $\omega_{1},\omega_{2}\in\Omega\cX(6)$.

\medskip

{\it Step 2:} 
Note that $p_i^*\eta_E$ vanishes only at the six branch points of $p_i$. 

In particular neither $\omega_{1}$ nor $\omega_{2}$ vanish at $\Fix(\beta)\cup\Fix(\alpha^{2}\beta)$, as the two sets of fixed points are disjoint.

\medskip

{\it Step 3:} Choose $\lambda$ so that $E\cong E_{\lambda}=\left\{y^{2}=x(x-1)(x-\lambda)\right\}$ with the point at infinity as a distinguished point, and let $[P]=(A,\pm B)$ in these coordinates. Then we claim that $\omega_{1},\omega_{2}\in\Omega\cX(6)$ if and only if $3A-1-\lambda=0$.

In fact note that, in this case, the map $\varphi\colon E_\lambda\to\bP^{1}$ in the proof of~\autoref{prop:D8ellipticcurves} can be chosen to be $(x,y)\mapsto x-A$, and the points in $X_{(E_\lambda,[P])}$ outside of the branch loci of the maps $p_i$ can be seen as pairs of points
	\begin{align*}Q =&\left((x,\varepsilon_{1}\sqrt{x(x-1)(x-\lambda)}),(-x+2A,\varepsilon_{2}\i\sqrt{(x-2A)(x-2A+1)(x-2A+\lambda)})\right)\\
	&\in E_\lambda\times E_\lambda\,,
	\end{align*}
where $\varepsilon_{1},\varepsilon_{2}\in\{\pm 1\}$. Normalising $\eta_{E}=\d x/y$ and evaluating a local expression around $Q$ yields
	\begin{align*}
	\omega_{1}(Q) &= \frac{\varepsilon_{1}}{\sqrt{x(x-1)(x-\lambda)}}+\frac{\varepsilon_{2}}{\sqrt{(x-2A)(x-2A+1)(x-2A+\lambda)}} = \\
	&= \frac{\varepsilon_{1}\sqrt{(x-2A)(x-2A+1)(x-2A+\lambda)}\pm\varepsilon_{2}\sqrt{x(x-1)(x-\lambda)}}{\varepsilon_{1}\varepsilon_{2}\sqrt{x(x-1)(x-\lambda)(x-2A)(x-2A+1)(x-2A+\lambda)}}\,,
	\end{align*}
and similarly for $\omega_{2}$.

Now, comparing the addends in the numerator and taking squares one sees that the differential $\omega_{1}$ will vanish at $Q$, for (exactly) two choices $(\varepsilon_{1},\varepsilon_{2})$ and $(-\varepsilon_{1},-\varepsilon_{2})$, whenever
	\[ 2 {\left(3 A - \lambda - 1\right)} x^{2} - 4 A {\left(3 A - \lambda - 1\right)} x + 2A(2A-\lambda)(2A-1) = 0\,.\]
In particular, if (and only if) $3A-\lambda-1=0$ the differentials do not vanish in the affine part of $E_\lambda$, hence the zeroes of $\omega_{1}$ and $\omega_{2}$ must be at infinity, i.e. in $\Fix(\alpha)$, and Step 1 implies that there is only a single zero on $X_{(E_\lambda,[P])}$.

\medskip

{\it Step 4:} The point $P=(A,B)\in E_\lambda$ can be uniquely chosen as a non-$2$-torsion point subject to the condition $A=(\lambda+1)/3$ from above if and only if $\lambda\neq -1,1/2,2$.

Note that these three values of $\lambda$ give rise to the same elliptic curve, namely the square torus $E_{2}$. In particular, for all $P \in E_{2}\setminus E_{2}[2]$ the curve $X_{(E_{2},[P])}$ has no $\alpha$-eigenform with a single zero.

\medskip

Therefore, for any $(E,O)\in\cM_{1,1}\setminus\{E_2\}$, there is a unique choice of $[P]$, such that the fibre product $X_{(E,[P])}$ admits a $D_8$ action together with an $\alpha$-eigenform that has a single $6$-fold zero.

\medskip

It remains to check when $\alpha$ can be extended, i.e. when there exists an $\alpha'\in\Aut X_{(E,[P])}$ that satisfies $\alpha\in\langle\alpha'\rangle$.

However, the proof of \autoref{prop:possibleorders} shows that this can happen only if $\alpha'$ is of order $12$. In this case, $(\alpha')^6=\alpha^2$ commutes with $\beta$, hence descends to an automorphism of order $6$ on the elliptic curve $X/\beta$ which must therefore be isomorphic to $E_{\zeta_6}$.

On the other hand, denote by $\psi\in\Aut E_{\zeta_{6}}$ the automorphism of order 6 on $E_{\zeta_6}$. It is easy to see that the automorphism $(Q_{1},Q_{2})\mapsto(\psi(Q_{2}),\psi^{4}(Q_{1}))$ on $E_{\zeta_{6}}\times E_{\zeta_{6}}$ restricts to an automorphism of order $12$, extending $\alpha$, on the curve of $\cX$ corresponding to this elliptic curve.
In fact, the corresponding fibre product is the curve $y^{6} = x (x+1)^{2} (x-1)^{2}$, see also \autoref{sec:e3}.
\end{proof}

Moreover, we have the following corollary.

\begin{cor}\label{cor:D8minimallocus}Let $X$ be a genus $4$ curve admitting an automorphism $\alpha$ of order $4$ with two fixed points. If additionally $(X,\omega)\in\Omega\W$ then $X\in\cX$ and $\omega=\omega_{1}$ or $\omega_{2}$. In particular, $\omega$ is an $\alpha$-eigenform and $(X,\omega)$ is a point of order $2$ on $\W$.
\end{cor}

To check which $(X,\omega_{i})\in\Omega\cX$ are on $\W$, we need to check when $\cP(X,\alpha^{2})$ admits real multiplication with $\omega_{i}$ as an eigenform. Note that $\beta^{*}$ interchanges $\omega_{1}$ and $\omega_{2}$, and therefore it is enough to focus on one of the two eigenforms.

First, we need the following explicit description of the endomorphism ring of the Prym variety. Recall that the endomorphism ring $\End(E)$ of an elliptic curve is either $\bZ$ or an order in an imaginary quadratic field.

\begin{lemma}\label{lem:etabasis}
Let $(\eta_1,\eta_2)$ be the product basis of $\Omega X^-$ as in~\autoref{D8pullbackdiffs}. Then
	\[\End(\cP(X,\alpha^2))=M_2(\End(E))\,,\] where $E\cong X/\beta$. Self-adjoint endomorphisms correspond to matrices satisfying $M^{T}=M^{\sigma}$, where $M^{\sigma}$ denotes conjugation by the non-trivial Galois automorphism of $\End(E)$ on each entry.

Moreover, $\omega_1$ corresponds to the $\eta$-representation $(1,\ \i)$ and $\omega_2$ to the representation $(1,\ -\i)$.
\end{lemma}

\begin{proof}
The first part of the lemma follows immediately from \autoref{cor:D8Prym}. The claim about the eigenforms follows from~\autoref{eq:eigencoordinates}.
\end{proof}

We now have all the ingredients assembled to prove the formula for the points of order $2$.

\begin{proof}[Proof of~\autoref{thm:e2thm}]
Recall that $X$ classifies an orbifold point of order $2$ on $\W$ if and only if $\cP(X,\alpha^{2})$ admits proper self-adjoint real multiplication with the $\alpha$-eigenforms $\omega_{1}$ or $\omega_{2}$ as an eigenform. Since these are interchanged by $\beta$, it is enough to focus our attention on $\omega_{1}$.

Again, we set $E=X/\beta$. Note that, by \autoref{thm:cXfamily}, $E$ must not be isomorphic to $E_{\zeta_6}$.

Assume that $D$ is not a square. Now, in the $\eta$-basis of $E\times E$, the form $\omega_{1}$ has the representation $(1,\ \i)$ (cf. \autoref{lem:etabasis}). In other words, $(X,\omega_{1})\in\Omega\cX$ is an orbifold point on $\W$ if and only if there exists $T_D\in M_2(\End(E))$, where
\[T_D=\begin{cases}\begin{pmatrix}\frac{1}{2}&-\frac{\sqrt{-D}}{2}\\\frac{\sqrt{-D}}{2}&\frac{1}{2}\end{pmatrix},&\text{ if $D\equiv 1\mod 4$, and}\\
\begin{pmatrix}0&-\frac{\sqrt{-D}}{2}\\\frac{\sqrt{-D}}{2}&0\end{pmatrix},&\text{ if $D\equiv 0\mod 4$,}\end{cases}\]
while there is no $T_{D'}\in M_2(\End(E))$ for $D=f^2D'$.

As $\End(E)$ is integral over $\bZ$ and $\nicefrac{1}{2}$ is not, the case $D\equiv 1\mod 4$ can never occur. The other case occurs whenever $\nicefrac{\sqrt{-D}}{2}\in\End(E)$, and this happens if and only if $E$ has complex multiplication by the order $\cO_{-D}\subset\bQ(\sqrt{-D})$.

To determine precisely which orders $\cO_{-C}$ contain such a maximal $T_{D}$, note that, by definition, $\nicefrac{\sqrt{-D}}{2}\in\cO_{-C}$ if and only if $D=b^2C$ for some integer $b$. Moreover, $C$ must be congruent with $0$ or $3\mod 4$ so that $-C$ is a discriminant.

For $b>2$ the action is never proper, and therefore we can assume $b=1$ or $2$.

The case $b=1$ implies that elliptic curves $E$ not isomorphic to $E_{\zeta_6}$ admitting complex multiplication by $\cO_{-D}$ always determine an orbifold point of order $2$ on $\W$. 

As for $b=2$, there are several options. If $\nicefrac{D}{4}\equiv 1 \mod 4$, then $-\nicefrac{D}{4}\equiv 3\mod 4$ is not a discriminant. If, however, $C=\nicefrac{D}{4}\equiv 3\mod 4$, then $-C$ is a discriminant and complex multiplication by $\cO_{-C}=\cO_{\nicefrac{-D}{4}}$ also gives proper real multiplication by $\cO_D$ on the Prym part. Finally, if $C=\nicefrac{D}{4}\equiv 0\mod 4$, then $-C$ is a discriminant but the Prym then admits real multiplication by $\cO_C$, hence the real multiplication by $\cO_D$ is not proper in these cases.

Moreover, observe that $E\cong E_{\zeta_6}$ if and only if $C=3$, i.e. $D=12$. On the other hand, if $D=12$, there exists precisely one elliptic curve with proper complex multiplication by $\cO_{-12}$ and hence $W_{12}(6)$ admits one point of order $2$ and one point of order $6$.

Finally, as it is well-known that there are $h(-C)$ elliptic curves admitting complex multiplication by $\cO_{-C}$, this proves the result.

\medskip

For the square discriminant case $D=f^{2}$, one can follow the same reasoning as above and use the fact that $\cO_{D}=\bZ[T]/(T^2-fT)$ to deduce that the  generator $T\in M_2(\End(E))$ must agree with
\[T=\begin{pmatrix}\frac{f}{2}& -\i\frac{f}{2}\\ \i\frac{f}{2}&\frac{f}{2}\end{pmatrix}\,\]
and an analysis similar to the one above proves the theorem.
\end{proof}

\section{Points of Order \texorpdfstring{$3$}{3}}\label{sec:e3}

In this section we prove the formula for the orbifold points of order $3$ on $\W$.

Recall the numbers
\[e_3(D)=%
\# \{a,i,j\in\bZ : a^2 + 3j^2 + (2i-j)^2=D,\ \gcd(a,i,j)=1\}/12.
\]
We have the following description of the orbifold points of order $3$.

\begin{theorem}\label{thm:e3thm}
Let $D\neq 12$ be a 
 positive discriminant. Then $\W$ has $e_3(D)$ orbifold points of order $3$.
\end{theorem}

To describe the points of order $3$ on $\W$, we again describe the intersection with the locus of curves with a fixed type of automorphism.

\subsection*{Curves admitting an automorphism of order 6}
By \autoref{prop:possibleorders}, for an $(X,\omega)$ to parametrise a point of order $3$ on $\W$ the curve $X$ must necessarily admit an automorphism $\alpha$ of order six with two fixed points and two orbits of length $2$ admitting $\omega$ as an eigenform. Note that in particular $X/\alpha$ has genus $0$. Cyclic covers of the projective line have been thoroughly studied by several authors (see for example~\cite{rohde,bouwhabil}), see also~\cite{TTZ} for a brief summary of the facts required here).

Now, there are two families of cyclic covers of $\mathbb{P}^{1}$ of degree 6 with the given branching data, namely:
	\[\cY_{t}\ \colon\ y^{6} = x (x-1)^{2} (x-t)^{2}\,,\quad t\in\bP^{1}\backslash\{0,1,\infty\} \]
and
	\[\cZ_{t}\ \colon\ y^{6} = x (x-1)^{2} (x-t)^{4}\,,\quad t\in\bP^{1}\backslash\{0,1,\infty\}\,.\]


Denote by $\alpha=\alpha_t$ the automorphisms $(x,y)\mapsto (x,\zeta_{6}y)$ of order $6$ on $\cY_{t}$ and on $\cZ_{t}$. Note that both $\cY_{t}/\alpha^{3}$ and $\cZ_{t}/\alpha^{3}$ have genus $2$, so $\rho=\alpha^3$ is actually a Prym involution. 

The following proposition shows immediately that no member of the $\cZ$ family can belong to a Teichmüller curve $\W$.

\begin{lemma}\label{lem:Zfamily}The space $\Omega\cZ^{-}$ is disjoint from the minimal stratum $\Omega\cM_{4}(6)$.
\end{lemma}

\begin{proof}It is easy to check (see for example~\cite{bouwhabil}) that for each $t$ the space $\Omega\cZ_{t}^{-}$ is generated by the differentials
	\[\xi_{1}=\frac{y\d x}{x(x-1)(x-t)}\quad\mbox{and}\quad \xi_{2}=\frac{\d x}{y}\,.\]
They both lie in the stratum $\Omega\cM_{4}(1^{2},4)$. In fact, a local calculation shows that
\begin{align*}
\div\xi_1&=4P_1+R'_1+R'_2\quad\text{and}\\
\div\xi_2&=4P_2+R''_1+R''_2,
\end{align*}
where the $P_i$ are the two fixed points of $\alpha$ and $\{R'_i\}$ and $\{R''_i\}$ are the $\alpha$-orbits of length $2$.

Now, any element of $\Omega\cZ_{t}^{-}$ different from the generators can be written as a linear combination $\xi=a\,\xi_{1}+b\,\xi_{2}$. But, since the $\xi_i$ vanish at different points, such a differential can never have a zero at any point of $\Fix(\alpha)$, nor at any point in the two $\alpha$-orbits of length $2$. As a consequence $\xi\in\Omega \cZ_{t}^{-}(1^{6})$ and the result follows.
\end{proof}


The following lemma detects which fibres of the $\cY$ family are isomorphic, together with the special fibre having a larger automorphism group.

\begin{lemma}\label{lem:parametert}The isomorphism $z\mapsto 1/z$ of $\bP^{1}$ lifts to an isomorphism $\cY_{t}\cong\cY_{\nicefrac{1}{t}}$ for each $t\in\bP^{1}\backslash\{0,1,\infty\}$. 

In particular, at the fixed point, the automorphism $\alpha_{-1}$ of the curve $\cY_{-1}$ extends to an automorphism $\gamma\colon(x,y)\mapsto (1/x,y/x)$ of order $12$.
\end{lemma}

\begin{proof}
As the curve is given in coordinates explicitly as a cyclic cover of $\bP^1$, this is a straight-forward calculation.
\end{proof}

The intersections of $\cY$ and $\W$ will give the orbifold points of order $3$ on $\W$. To make this statement more precise, we begin by the following observation.

\begin{prop}\label{prop:C6eigenbasis}
For each $t$ the space $\Omega\cY_{t}^{-}$ is generated by the $\alpha$-eigenforms
	\[\omega_{1}=\frac{y\d x}{x(x-1)(x-t)}\quad\mbox{and}\quad\omega_{2}=\frac{-y\d x}{\sqrt{t}(x-1)(x-t)}\,. \]
Up to scale, the only differentials in $\Omega\cY_{t}^{-}(6)$ are $\omega_{1}$ and $\omega_{2}$.
\end{prop}

\begin{proof}The local expressions show that these differentials are holomorphic for all $t$. They obviously span the $\alpha$-eigenspace of eigenvalue $\zeta_{6}$ and therefore generate $\Omega\cY_{t}^{-}$ (cf.~\cite{bouwhabil}).

It is easy to see that $\omega_{1}$ (resp. $\omega_{2}$) has a single zero at the single point at infinity (has a single zero at $(0,0)$). Now, for every $a\neq 0$ the zeroes of the differential
	\[\omega_{a}\coloneqq\omega_{1}+a\omega_{2}= \frac{y\d x}{(x-1)(x-t)} \left(\frac{1}{x}+a\right)\]
are located at the points with $x$-coordinate $\nicefrac{-1}{a}$. They are either six simple zeroes if $a\neq -1,\nicefrac{-1}{t}$, or three zeroes of order 2 otherwise.
\end{proof}

\begin{rem}
Note that, in contrast to the family $\cX$ of curves with a $D_8$ action, the $\alpha$-eigenspace inside $\Omega\cY_t^-$ is in fact $2$-dimensional. However, we will only be interested in the two $1$-dimensional subspaces of eigenforms with a single zero.
\end{rem}

Because of the flat picture of the differentials $(\cY_{t},\omega_{i})$, we will call $\cY$ the \emph{\Yfamily{}} (see~\autoref{sec:flat}). Note that $(\omega_{1},\omega_{2})$ yields an $\alpha$-eigenbasis of $\Omega\cY_t^{-}$. The following is a consequence of~\autoref{lem:Zfamily} and~\autoref{prop:C6eigenbasis}.

\begin{cor}\label{cor:C6minimallocus}Let $X$ be a genus $4$ curve admitting an automorphism $\alpha$ of order $6$ with two fixed points and two orbits of length $2$. If $(X,\omega)\in\Omega X^{-}(6)$ then $X\in\cY$ and $\omega=\omega_{1}$ or $\omega_{2}$. In particular, $\omega$ is an $\alpha$-eigenform.
\end{cor}

\begin{cor}\label{e3orbicrit} A flat surface $(X,\omega)$ parametrising a point on $\W$ corresponds to an orbifold point of order $3$ if and only if there is some $t\in\bP^1\backslash\{0,1,-1,\infty\}$ such that $X\cong\cY_t$ and $[\omega]=[\omega_1]$ or $[\omega]=[\omega_2]$.

It corresponds to an orbifold point of order $6$ if and only if $X\cong\cY_{-1}$ and $[\omega]=[\omega_1]$ or $[\omega]=[\omega_2]$ 
\end{cor}

\begin{proof}This is a consequence of~\autoref{prop:possibleorders},~\autoref{lem:Zfamily} and~\autoref{prop:C6eigenbasis}.
\end{proof}

We must therefore analyse when the Prym part of $\cY_t$ admits real multiplication. Recall that the elliptic curve $E_\zeta$, where $\zeta\coloneqq \exp(2\pi\i/6)$, is the only elliptic curve admitting an automorphism of order $6$ fixing the base point. It corresponds to the hexagonal lattice, i.e.
\[E_\zeta\cong\bC/\Lambda_\zeta,\quad\text{with}\quad\Lambda_\zeta=\bZ\oplus\zeta\bZ\,,\]

Next, we collect some useful observations.

\begin{lemma}\label{cor:C6Prym}
Any curve $\cY_t$ admits an involution $\beta$ commuting with $\alpha$, i.e. such that $\langle\alpha,\beta\rangle\cong C_6\times C_2$. Moreover, one has $\cP(\cY_t,\rho)\cong E_\zeta\times E_\zeta$.

The general member $\cZ_t$ of the $\cZ$ family has an automorphism group equal to $C_6$.
\end{lemma}

\begin{proof}By Theorems 1 and 2 in~\cite{Sin} there is only one Fuchsian group containing a generic Fuchsian group of signature $(0;3,3,6,6)$. The signature of such supergroup is $(0;2,2,3,6)$, and the inclusion is of index 2 and therefore normal. As a consequence, the automorphism group of any general fibre in the $\cY$ family or the $\cZ$ family is at most an extension of index two of $C_{6}$.

In the case of $\cY_{t}$, the inclusion induces an extra automorphism $\beta\coloneqq\beta_{t}$, given by $(x,y)\mapsto(x/t,\sqrt{t}\,y/x)$.

In particular $\alpha^{3}$ and $\beta$ generate a Klein four-group such that the quotients $\cY_{t}/\beta$ and $\cY_{t}/\alpha^{3}\beta$ have genus $1$. Therefore they satisfy the conditions of \autoref{prop:PrymIsomorphism} and $\cP(\cY_t,\rho)\cong \cY_{t}/\beta\times\cY_{t}/\alpha^{3}\beta$. Since $\alpha$ induces an automorphism $\psi$ of order $6$ on both $\cY_{t}/\beta$ and $\cY_{t}/\alpha^{3}\beta$, they are necessarily isomorphic to the elliptic curve $E_{\zeta}$.

As for the $\cZ$ family, any such automorphism would induce an automorphism of $\cZ_{t}/\alpha\cong\bP^{1}$ permuting orbifold points of the same order. Since the exponents at $0$ and $\infty$ and at $1$ and $t$ are different, there cannot be such an automorphism.
\end{proof}

We will write again $p_{1}\colon\cY_{t} \to \cY_{t}/\beta$ and $p_{2}\colon\cY_{t}\to \cY_{t}/\alpha^{3}\beta$ for the corresponding projections. The following lemma gives an explicit formula for these two maps that will be needed later to compute the explicit pullbacks of the differentials on $E_\zeta$.

\begin{lemma}\label{lem:C6ellipticcurve}Consider the Weierstra{\ss} equation $\{v^{2}=u^{3}-1\}$ defining $E_{\zeta}$. In this model, the maps $p_{1}$ and $p_{2}$ are given by
\begin{align*}
	p_{1}\colon \cY_{t} 	& \to		 E_{\zeta} \\
			 (x,y)		& \mapsto 	\left(\dfrac{-1}{(1+\sqrt{t})^{\nicefrac{2}{3}}} \dfrac{(x-1)(x-t)}{y^{2}} , \dfrac{\i}{(1+\sqrt{t})} \dfrac{(x-1)(x-t)(x+\sqrt{t})}{y^{3}}\right)\,, \\[12pt]
	p_{2}\colon \cY_{t} 	& \to		 E_{\zeta} \\
			 (x,y)		& \mapsto 	 \left(\dfrac{-1}{(1-\sqrt{t})^{\nicefrac{2}{3}}} \dfrac{(x-1)(x-t)}{y^{2}} , \dfrac{\i}{(1-\sqrt{t})} \dfrac{(x-1)(x-t)(x-\sqrt{t})}{y^{3}}\right)\,.
\end{align*}
These maps are only unique up to composition with (a power of) $\alpha$.
\end{lemma}

\begin{proof}
The map $\cY_{t}\to\cY_{t}/\beta$ induces an isomorphism between the function field $\bC(\cY_{t}/\beta)$ and the subfield $\bC(\cY_{t})^{\langle\beta\rangle}\subset \bC(\cY_{t})$ fixed by $\beta^{*}$. This subfield is generated by the rational functions
	\[\widetilde{u}\coloneqq x+\beta(x)+2\sqrt{t}=\dfrac{(x+\sqrt{t})^{2}}{x}   \,,\quad \widetilde{v}\coloneqq y+\beta(y) =y\dfrac{x+\sqrt{t}}{x}\,.\]

Using the equation of $\cY_{t}$ it is easy to check that the generating functions $\widetilde{u}$ and $\widetilde{v}$ satisfy the relation $\widetilde{v}^{6}=\widetilde{u}^{3}(\widetilde{u}-c)^{2}$, where $c=(1+\sqrt{t})^{2}$. One can then check the ramification points of the degree 6 function $(\widetilde{u},\widetilde{v})\mapsto \widetilde{u}$ and easily deduce the isomorphism
	\[\begin{array}{cll}
	\widetilde{E}\,:\ \widetilde{v}^{6}=\widetilde{u}^{3}(\widetilde{u}-c)^{2} & \to & E_{\zeta}\,:\ v^{2}=u^{3}-1 \\
	(\widetilde{u},\widetilde{v}) & \mapsto & (u,v)=\left(\dfrac{-\widetilde{u}(\widetilde{u}-c)}{c^{\nicefrac{1}{3}}\, \widetilde{v}^ {2}}, \dfrac{\i\, \widetilde{v}^{3}}{c^{\nicefrac{1}{2}}\, \widetilde{u}(\widetilde{u}-c)} \right)
	\end{array}\]

Finally, replacing $\widetilde{u}$ and $\widetilde{v}$ by their values in terms of the coordinates $x$ and $y$, one gets the formula for $p_{1}$.

The same argument replacing $\beta$ by $\alpha^{3}\beta$ yields the result for $p_{2}$.
\end{proof}

\subsection*{Fibre Products}
Similarly to the case of the $D_{8}$ family, one can also construct the \Yfamily{} $\cY$ of genus $4$ curves with a $C_{6}\times C_{2}$ action as a certain family of fibre products over two isomorphic elliptic curves. In order to do so, let $\psi$ denote the automorphism of order $6$ on $E_{\zeta}$ 
and consider the following diagram:
\[\begin{tikzpicture}
\matrix (m) [matrix of math nodes, row sep=1.75em, column sep=3.5em, text height=1.5ex, text depth=0.25ex]
{ X & X/\beta\cong E_\zeta\\
X/\alpha^{3} & X/\langle\alpha^3,\beta\rangle\cong\bP^1 \\
\bP^1\cong X/\alpha & X/\langle\alpha,\beta\rangle\cong\bP^1 \\
};
\path[->,font=\scriptsize]
(m-1-1) edge [above] node {$p_1$} (m-1-2)
(m-1-2) edge [right] node {$\varphi$} (m-2-2)
(m-1-1) edge (m-2-1)
(m-2-1) edge (m-3-1)
(m-2-1) edge (m-2-2)
(m-2-2) edge (m-3-2)
(m-3-1) edge (m-3-2)
;
\end{tikzpicture}\]
Clearly, $\overline{\beta}$ is the hyperelliptic involution on $X/\alpha^3$. On $X/\alpha\cong\bP^1$ the involution $\overline{\beta}$ has two fixed points and the preimages of these points give the six Weierstraß points on $X/\alpha^3$. Moreover, $X\to X/\alpha^3$ is ramified only over the two fixed points of $\alpha$, while the map $X/\alpha^3\to X/\alpha$ also branches at $R'$ and $R''$, the preimages (on $X$) being $\{R'_1,R'_2\}$ and $\{R''_1,R''_2\}$, respectively.

Now, $\alpha$ and $\beta$ have no common fixed points, hence the image of the (two) fixed points of $\alpha$ on $X$ gives the (unique) fixed point $O$ of $\psi=\overline{\alpha}$ on $E_\zeta$. Additionally, $\overline{\beta}$ interchanges $R'$ and $R''$, hence we may name the fibres such that the images $R_1$ of $\{R'_1,R''_1\}$ and $R_2$ of $\{R'_2,R''_2\}$ form the unique $\psi$-orbit of order $3$ on $E_\zeta$.

On the other hand, the six Weierstraß points of $X/\alpha^3$ have $12$ preimages on $X$ with $\beta$ acting on each fibre. Three fibres form the six fixed points of $\beta$ on $X$, i.e.\ the branch points of $p_1$, while the other three give the fixed points of $\alpha^3\beta$, which are equivalently the fixed points of the elliptic involution $\phi=\overline{\alpha}^3$ on $X/\beta$, i.e.\ the three $2$-torsion points. The situation is exactly reversed for the projection $p_2\colon X\to X/\alpha^3\beta\cong E_\zeta$.

Finally, note that in \autoref{lem:C6ellipticcurve} the coordinates on $E_\zeta$ were chosen such that the projection $\varphi\colon E_{\zeta} \to \bP^{1}\cong E_{\zeta}/\phi$ to the quotient by the elliptic involution maps $O$ to $\infty$ and both $R_1$ and $R_2$ to $0$. Observe that $\psi$ then descends to an automorphism of order $3$ on the quotient that fixes $0$ and $\infty$. In particular, we can assume $\varphi(\psi(S))=\zeta_{6}^2\varphi(S)$, for each $S\in E_{\zeta}$.

%

Now, for each $P\in E_{\zeta}^{*} \coloneqq E_{\zeta}\backslash (E_{\zeta}[2] \cup \{R_{1},R_{2}\})$ consider the map $\varphi_{P}\colon E_{\zeta} \to \bP^{1}$, $Q\mapsto \varphi(P)\cdot\varphi(Q)$. We define $Y=Y_{P}$ as the fibre product of the diagram 
\[E_{\zeta}\xrightarrow{\ \varphi\ } \mathbb{P}^{1} \xleftarrow{\,\varphi_{P}}E_{\zeta}.\]
This fibre product admits a group of automorphisms isomorphic to $C_{6}\times C_{2}$ given by the restriction of the following automorphisms of $E_{\zeta}\times E_{\zeta}$:
	\[\alpha(Q_{1},Q_{2})=(\psi(Q_{1}),\psi(Q_{2}))\,,\quad\beta(Q_{1},Q_{2})=(Q_{1},\psi^{3}(Q_{2}))\,.\]
By \autoref{cor:C6Prym}, every $Y_P$ is therefore a fibre of the $\cY$ family.

\begin{prop}\label{prop:C6ellipticcurves}
The map $P\mapsto Y_{P}$ 
gives a 6-to-1 map between the set of elliptic pairs of points $E_{\zeta}^{*}/\phi$ and the fibres of $\cY$.

It descends to a 2-to-1 map between the set $E_{\zeta}^{*}/\psi$ of regular orbits of $\psi$ and the fibres of $\cY$. Moreover, the only ramification value of this map corresponds to the curve $\cY_{-1}$ admitting an automorphism of order 12.
\end{prop}

\begin{proof}Let $P\in E_{\zeta}^{*}$. Note that the construction does not depend on the choice of $\{P,\phi(P)\}$. In fact, even for any choice of a different point in the orbit $\{\psi^{j}(P)\}_{j=0}^{5}$ the automorphism of $E_{\zeta}\times E_{\zeta}$ given by $(Q_{1},Q_{2}) \mapsto (Q_{1},\psi^{-j}(Q_{2}))$ 
induces an isomorphism between $Y_{P}$ and $Y_{\psi^{j}(P)}$.

Now, for the point $P'\in E_{\zeta}^{*}$ such that $\phi(P')=1/\phi(P)$, the automorphism $(Q_{1},Q_{2}) \mapsto (Q_{2},Q_{1})$ induces an isomorphism between $Y_{P}$ and $Y_{P'}$.

On the other hand, for any $Y\in\cY$ take $x\in\Fix(\beta)$ and write $P=[x]\in Y/\beta \cong E_{\zeta}$ for its image in the quotient. It is straightforward to check that $Y\cong Y_{P}$. Any other choice of $x\in \Fix(\beta)$ or $x\in \Fix(\alpha^{3}\beta)$ determines different points in $\{\psi^{j}(P), \psi^{j}(P')\}_{j=0}^{5}$, defining the same fibre product.
\end{proof}

\begin{rem}
Note that the action of $\psi$ on the point $P$ corresponds to the action of $\alpha$ on the maps $p_i$ mentioned in \autoref{lem:C6ellipticcurve}. The remaining factor of $2$ comes from the (generic) identification of $\cY_t$ with $\cY_{\nicefrac{1}{t}}$
\end{rem}

\subsection*{Eigenforms with a single zero}
%
By \autoref{cor:C6Prym}, all Prym varieties in the $\cY$ family are isomorphic. To understand $\End(\cP(\cY_t,\rho))$, where $\rho=\alpha^3$, denote by $(\eta_1,\eta_2)$ again the product basis of $\Omega\cY_t^{-}$ given by
\begin{equation}\label{C6pullbackdiffs}
\eta_i=p_i^*\eta_E,\quad\text{for $i=1,2$.}
\end{equation}
%
%
%
It is well known that $\cO_\zeta\coloneqq\End(E_\zeta)=\bZ\oplus\bZ\zeta_6^2$ are the \emph{Eisenstein integers}. 

\begin{lemma}\label{C6EndPrym}
Let $(\eta_1,\eta_2)$ be the product basis of $\Omega\cY_{t}^-$ from~\autoref{C6pullbackdiffs}. Then
	\[\End(\cP(\cY_t,\alpha^{3}))=M_2(\End(E_\zeta))=M_2(\cO_\zeta).\]	
Self-adjoint endomorphisms correspond to matrices $M^{T}=M^{\sigma}$, where $M^{\sigma}$ denotes conjugation by the non-trivial Galois automorphism of $\cO_\zeta$ on each entry.
\end{lemma}

\begin{proof}
This is an immediate consequence of \autoref{cor:C6Prym}.
\end{proof}

While the product basis gives an easy understanding of the endomorphism ring, and while in fact any differential in $\Omega\cY_t^-$ is an $\alpha$-eigendifferential, we are interested in $\alpha$-eigendifferentials \emph{with a single zero} that are also eigenforms for real multiplication of the Prym variety. By \autoref{prop:C6eigenbasis}, these are precisely the differentials $\omega_1$ and $\omega_2$ on $\cY_t$.

To check whether $\omega_1$ or $\omega_2$ are eigenforms for real multiplication, we must therefore keep track of these differentials in the product basis. For this, we set
\[\mu\coloneqq\mu_t\coloneqq\left(\frac{1-\sqrt{t}}{1+\sqrt{t}}\right)^{\nicefrac{1}{3}}.\]
%
The relationship between the $\alpha$-eigenbasis and the product basis can be summarised as follows:

\begin{lemma}\label{C6BaseChange}
Denote by $(\eta_1,\eta_2)$ the product basis. Then
	\[ [\omega_1]=[-\mu_t\,\eta_1+\eta_2]\,,\quad%
	[\omega_2]=[\mu_t\,\eta_1+\eta_2]\]
gives an $\alpha$-eigenbasis with $\omega_i$ having each a single zero on $\cY_t$.

In particular, for each isomorphism class of curves $[Y]\in\cY$, with $[Y]\neq [\cY_{-1}]$, there exist 12 elements $\mu_t\in\bC^{*}$ such that $\mu_t\,\eta_{1}+\eta_{2}$ are precisely the $\alpha$-eigendifferentials with a single zero on $[Y]$.

On the curves $[\cY_{-1}]$ there are six different values of $\mu$ giving  eigendifferentials with a single zero.
\end{lemma}

\begin{proof}
The differential $\omega_{1}+\omega_{2}$ (resp. $\omega_{1}-\omega_{2}$) is $\beta$-invariant (resp. $\alpha^{3}\beta$-invariant). Therefore there exist $k_{1},k_{2}$ such that $\omega_{1}+\omega_{2}=k_{1}\,\eta_{1}$ and $\omega_{1}-\omega_{2}=k_{2}\,\eta_{2}$, where $\eta_{E}$ is a fixed differential on $E_{\zeta}$.

In particular
	\[\frac{-x}{\sqrt{t}}=\frac{\omega_{2}}{\omega_{1}} = \frac{k_{1}\,\eta_{1}-k_{2}\,\eta_{2}}{k_{1}\,\eta_{1}+k_{2}\,\eta_{2}}\,.\]

One can solve for $\nicefrac{k_{1}}{k_{2}}$ to get
	\[\frac{k_{1}}{k_{2}} = -\frac{(x-\sqrt{t})\eta_{2}}{(x+\sqrt{t})\eta_{1}}\,,\]
and then, using~\autoref{lem:C6ellipticcurve} and choosing $\eta_{E}=\d u/v$ in that model,
	\[\frac{k_{1}}{k_{2}} = -\left(\frac{1-\sqrt{t}}{1+\sqrt{t}}\right)^{\nicefrac{1}{3}}=-\mu\,.\]

Now, solving $t$ in terms of $\mu$ gives
	\[t=\left(\frac{\mu^{3}-1}{\mu^{3}+1}\right)^{2}\,,\]
and~\autoref{lem:parametert} implies the rest of the claims.
\end{proof}

Note that every value of $\mu$ gives two eigendifferentials with single zeros on (generically) two different fibres of $\cY$, which are identified by six different isomorphisms.

\begin{lemma}\label{flatclass}
For $t\neq -1$ we have that in $\bP\Omega\cY(6)$
\[(\cY_t,\omega_1)\cong(\cY_t,\omega_2)\cong(\cY_{\nicefrac{1}{t}},\omega_1)\cong(\cY_{\nicefrac{1}{t}},\omega_2)\]
as flat surfaces and $(\cY_t,\omega_1)\not\cong(X,\omega)$ for all other $(X,\omega)\in\bP\Omega\cY(6)$.
\end{lemma}

\begin{proof}
This is clear by \autoref{prop:C6eigenbasis}, \autoref{lem:parametert} and the fact that $\beta$ interchanges the classes of $\omega_1$ and $\omega_2$.
\end{proof}

In particular, we do not have to distinguish between the classes of $\omega_1$ and $\omega_2$.
This relationship becomes more explicit when expressed in the fibre product construction.


\begin{prop}\label{prop:cYfamily}Let $(\mu\,\eta_{E},\eta_{E})\in\Omega E_{\zeta}\times \Omega E_{\zeta}$, $\mu\neq 0$ and let $P=(A,B)\in E_{\zeta}^{*}$. The corresponding $\alpha$-eigendifferential $\mu\,\eta_{1}+\eta_{2}$ 
 on $Y_{P}$ has a single zero at a fixed point of $\alpha$ if and only if $A=\mu^{2}$.

In particular, this induces a 12-to-1 map 
	\[\begin{array}{lll}
	\mathbb{C}^{*} & \to & \bP\,\Omega\cY(6) \\
	\mu & \mapsto & [(Y_{P}, [\mu\,\eta_{1}+\eta_{2}])]\,, 
	\end{array}\]
where $[P]=(\mu^{2},\pm\sqrt{\mu^{6}-1})$ is an elliptic pair on $E_\zeta$.

\end{prop}

\begin{proof}For each $P\in E_{\zeta}^{*}$ we will consider the differentials $\eta_{1}=p_{1}^{*}\eta_{E}$ and $\eta_{2}=p_{2}^{*}\eta_{E}$ on $Y_{P}$. The proof of this theorem will proceed in a similar way to the proof of~\autoref{thm:cXfamily} up until Step 3.

\medskip

{\it Step 1:} The $\alpha$-eigenforms in $\Omega Y_{P}^{-}$ can have zeroes at most at one of the fixed points of $\alpha$. 

Otherwise, every differential in $\Omega Y_{P}^{-}$ would vanish at both fixed points of $\alpha$. In particular, so would $\eta_{1}$ and $\eta_{2}$, but the maps $p_i$ are unramified at $\Fix(\alpha)$ and we know that $\eta_{E}$ has no zeroes in $E$. 

Again, zeroes of $\alpha$-eigenforms must be permuted by $\alpha$, the orbits of which have length $1$, $2$ or $6$. This immediately implies that $\alpha$-eigenforms lie either in $\Omega Y_{P}(1^{6})$ if the zeroes are located at regular points, in $\Omega Y_{P}(6)$ if it only has zeroes at a fixed point of $\alpha$, or in $\Omega Y_{P}(3^2)$ if it has zeroes at the two points of the orbit of length $2$ (see the proof of~\autoref{prop:C6eigenbasis}). Again, we just need to prove that $\omega_{1},\omega_{2}\in\Omega Y_{P}(6)$.

\medskip

{\it Step 2:} Again, $p_i^*\eta_E$ vanishes only at the six branch points of $p_i$. In particular both $\eta_{1}$ and $\eta_{2}$ lie in $\Omega Y_{P}(1^{6})$.

\medskip

{\it Step 3:} Let $E_{\zeta}\cong \left\{v^{2}=(u^{3}-1)\right\}$ with the point at infinity as a distinguished point, and let $(\mu\,\eta_{E},\eta_{E})\in\Omega E_{\zeta}\times\Omega E_{\zeta}$. We claim that, given a point $P=(A,B)$ in these coordinates, the differential $\mu\,\eta_{1}+\eta_{2}$ on $Y_{P}$ has a single zero if and only if $A=\mu^{2}$.

Note that we can normalise $\varphi:E_{\zeta}\to\bP^{1}$ to be $(u,v)\mapsto u$. By construction of $Y_{P}$ as the fibre product of the maps $\varphi,\varphi_{P}:E_{\zeta}\to\bP^{1}$, points in $Y_{P}$ outside of the branch loci of the maps $p_i$ can then be seen as pairs
	\[Q = \left(\left(u,\varepsilon_{1}\sqrt{u^{3}-1}\right),\left(\frac{u}{A},\varepsilon_{2}\sqrt{\frac{u^{3}-A^{3}}{A^{3}}}\right)\right)\in E_{\zeta}\times E_{\zeta}\,,\]
where $\varepsilon_{1},\varepsilon_{2}\in\{\pm 1\}$. Normalising $\eta_{E}=\d x/y$, and evaluating locally around $Q$ yields
	\[\mu\,\eta_{1}+\eta_{2}(Q) = \frac{\varepsilon_{1}\mu}{\sqrt{u^{3}-1}}+\frac{\varepsilon_{2}\sqrt{A^{3}}}{A\sqrt{u^{3}-A^{3}}} =
\frac{\varepsilon_{1}\mu A\sqrt{u^{3}-A^{3}}+\varepsilon_{2}\sqrt{A^{3}(u^{3}-1)}}{A\sqrt{(u^{3}-1)(u^{3}-A^{3})}}\,.\]
Comparing again the addends in the numerator and taking squares, one sees that this differential vanishes at $Q$ (for two choices $(\varepsilon_{1},\varepsilon_{2})$ and $(-\varepsilon_{1},-\varepsilon_{2})$) whenever 
	\[ u^{3}=A\cdot\frac{\mu^{2}A^{2}-1}{\mu^{2}-A}\,.\]
In particular, whenever the right-hand side is different from $0$, $1$ and $\infty$ one has that the differential $\mu\, \eta_{1}+\eta_{2}$  necessarily has $6$ simple zeroes. The case $u^{3}=1$ corresponds to $\mu=0$, which has been treated in Step 2. The case $u=0$ corresponds to $A=\pm \nicefrac{1}{\mu}$ and yields the differentials with zeroes at the two points of the $\alpha$-orbit of length $2$.

Finally, if $A=\mu^{2}$ the zeroes of the differential must be in $\Fix(\alpha)$, and Step~1 then implies that there is a single zero.

\medskip

As a consequence, to each $\mu\in\bC^{*}$ we can associate the elliptic pair of points 
\[[P]=(\mu^{2},\pm\sqrt{\mu^{6}-1})\in E_{\zeta}^{*},\] 
defining the curve $Y_{P}$ together with the differential with a single zero $\mu\,\eta_{1}+ \eta_{2}$. By \autoref{prop:C6ellipticcurves} and the fact that $\pm\mu$ give the same elliptic pair and, by \autoref{flatclass}, the same class of flat surfaces, the association is 12-to-1.
\end{proof}

We are now finally in a position to prove the formula for $e_3(D)$.

\begin{proof}[Proof of \autoref{thm:e3thm}]
First, let $D$ be a nonsquare discriminant and recall the order $\cO_D=\bZ\oplus T_{D}\bZ$ associated to $D$, where
\[T_{D}=\begin{cases}\frac{\sqrt{D}}{2},&D\equiv 0\mod 4,\\\frac{\sqrt{D}+1}{2},&D\equiv 1\mod 4.\end{cases}\]
Then, for $i=1,2$, $(\cY_t,\omega_i)$ lies on $\W$ if and only if $\cP(\cY_t,\rho)$ admits real multiplication with $\omega_i$ as an eigenform. By \autoref{C6EndPrym} and \autoref{C6BaseChange} this is equivalent to the existence of some self-adjoint matrix
\[A=\begin{pmatrix}a&b\\c&d\end{pmatrix}\in M_2(\cO_\zeta)\quad\text{such that}\quad A\cdot\begin{pmatrix}\pm\mu\\ 1\end{pmatrix}=T\cdot \begin{pmatrix}\pm\mu\\ 1\end{pmatrix}.\]
By \autoref{flatclass}, it suffices to consider $+\mu$.
Moreover, by self-adjointness, we have $c=b^\sigma$, the Galois conjugate in $\cO_\zeta$, and $a,d\in\bZ$. The eigenform condition  then yields
\[(a-T_{D})\mu+b=0\quad\text{and}\quad  b^\sigma\mu+d-T_{D}=0.\]
The first equation gives
\begin{equation*}
\mu=\frac{b}{T_{D}-a}
\end{equation*}
and substituting this into the second equation yields 
\[bb^\sigma - ad=T_{D}^2-(a+d)T_{D}.\]
First, we consider the case $D\equiv 0\mod 4$. Then this gives
\[bb^\sigma - ad = \frac{D}{4}-(a+d)\frac{\sqrt{D}}{2}.\]
As the right side of the equation must be an integer, we find $a=-d$ and hence
\begin{equation*}
D=4bb^\sigma + (2a)^2,\quad\text{for}\quad D\equiv 0\mod 4.
\end{equation*}
Similarly, for $D\equiv 1\mod 4$, we obtain $d=a-1$ and thus
\begin{equation*}
D=4bb^\sigma + (2a-1)^2,\quad\text{for}\quad D\equiv 1\mod 4.
\end{equation*}
It is well-known that the norm squared of an element in $\cO_\zeta$ is given by
\[bb^\sigma=i^2-ij+j^2=\frac{3j^2+(2i-j)^2}{4},\quad\text{for}\quad b=i+\zeta_6^2 j.\]
Hence, $\cP(\cY_t,\rho)$ admits a real multiplication by $\cO_D$ with $\omega_i$ as an eigenform for every $a,i,j\in\bZ$ such that
\[a^2 + 3j^2 + (2i-j)^2=D.\]
Clearly, this real multiplication is proper if and only if $\gcd(a,i,j)=1$.

By \autoref{C6BaseChange} or equivalently \autoref{prop:cYfamily}, this gives $12$ times the cardinality of points of order $3$.

\medskip

A similar analysis in the square discriminant case $D=f^{2}$ yields, with the same notation as above, $d=f-a$ and $bb^\sigma - a(f-a) = 0$. Multiplying by $4$ and adding $f^{2}$ to both sides of the equation one gets
\begin{equation*}
D=4bb^\sigma + (2a-f)^2\,,
\end{equation*}
and the same argument as above proves the result.
\end{proof}

\section{Points of Order \texorpdfstring{$5$}{5}}\label{sec:e5}

In this section we will find the orbifold points of order $5$ on the Teichm\"{u}ller curves $\W$.

\begin{theorem}\label{thm:e5thm} The Teichm\"{u}ller curve $W_{5}$ has one orbifold point of order $5$. For any other discriminant, $W_{D}$ has no orbifold points of order $5$.
\end{theorem}

\subsection*{Curves admitting an automorphism of order 10}

By~\autoref{prop:possibleorders}, flat surfaces $(X,\omega)$ parametrising a point of order $5$ on $\W$ will correspond to cyclic covers of degree $10$ of $\bP^{1}$ ramified over three points with ramification order $5$, $10$ and $10$. There are two such curves:
	\[ \cV\ \colon\ y^{10} = x (x-1)^{2}\,,\quad\mbox{and}\quad \cU\ \colon\ y^{10} = x (x-1)^{8}\,. \]

Calculations similar to the ones in the proof of~\autoref{lem:Zfamily} and~\autoref{prop:C6eigenbasis} give us the differentials with a single zero on these curves.

\begin{prop}\label{prop:C10eigenbasis}The space $\Omega\cV^{-}$ is generated by the $\alpha$-eigenforms
	\[\omega_{1}=\frac{y\d x}{x(x-1)}\,,\ \omega_{2}=\frac{y^{3}\d x}{x(x-1)}\,.\]
Up to scale, the only differential in $\Omega\cV^{-}(6)$ is $\omega_{1}$.

The space $\Omega\cU^{-}$ is disjoint from the minimal stratum $\Omega\cM_{4}(6)$.
\end{prop}

In particular one has the following corollary.

\begin{cor}\label{cor:C10minimallocus}Let $X$ be a genus $4$ curve admitting an automorphism $\alpha$ of order $10$ with two fixed points and an orbit of length $2$. If $(X,\omega)\in\Omega X^{-}(6)$ then $X=\cV$ and, up to scale, $\omega=\omega_{1}$. In particular, $\omega$ is an $\alpha$-eigenform.
\end{cor}

The action of $\alpha$ on $\cP(\cV,\alpha^{5})$ induces an embedding $\bQ(\zeta_{10})\hookrightarrow \End_{\bQ}(\cP(\cV,\alpha^{5}))$ and, in particular, determines an element $T_{5}=\alpha+\alpha^{-1}=(\sqrt{5}+1)/2$ for which $\omega_{1}$ is an eigenform.

\begin{proof}[Proof of~\autoref{thm:e5thm}]By the argument above and the maximality of $\cO_{5}$, the Prym variety $\cP(\cV,\alpha^{5})$ admits proper real multiplication by $\cO_{5}$ with $\omega_{1}$ as an eigenform.

Now Teichm\"{u}ller curves $\Omega W_{D}$ and $\Omega W_{E}$ are disjoint for different discriminants $D$ and $E$. Therefore, as $\omega_{1}$ is, up to scale, the only differential with a single zero on $\cV$, there can be no other $W_D$ with a point of order $5$.
\end{proof}

\section{Flat geometry of orbifold points}\label{sec:flat}

In this section we will describe, up to scale, the translation surfaces corresponding to the \Xfamily{} $\cX$, the \Yfamily{} $\cY$ and the curve $\cV$. We use the notion of $k$-differentials and $(1/k)$-translation structures, cf. \cite[\S 2.1, 2.3]{kdifferentials}.

Note that, whereas in the first two cases we have a $1$-dimensional family of flat surfaces, in the case of $\cV$ the construction is unique (cf.~\autoref{cor:C10minimallocus}). The case of a flat surface with a symmetry of order twelve, also unique, is given by $\cX\cap\cY$, the intersection of the \Xfamily{} and the \Yfamily{} (cf. \autoref{thm:cXfamily}).

\subsection*{Points of order \texorpdfstring{$2$}{2}}
We briefly describe the construction of flat surfaces $(\cX_{\kappa},\eta_{\kappa})\in \Omega\cX(6)$ (that is curves $X_{\kappa}$ with a four-fold symmetry $\alpha$ together with a differential $\eta_{\kappa}$ with a six-fold zero) in terms of a parameter $\kappa=\kappa(c,\theta)$.

By \autoref{prop:possibleorders}, the quotient $X_{\kappa}/\alpha$ is of genus $1$ with two fixed points. Therefore, a $4$-differential $\xi$ of genus $1$ with a zero and a pole, each of order $3$, at the two fixed points will have $(X_{\kappa},\eta_{\kappa})$ as a canonical cover, i.e.\ $\eta_{\kappa}^4=\pi^*\xi$, cf. \cite{kdifferentials}. The polygon corresponding to $\xi$ is given in \autoref{fig:C4onE} with an angle of $2\pi/4$ at the pole and $7\cdot 2\pi/4$ at the zero. Note that the three pairs of sides are identified by translation and rotation by angle $\pi/2$ and that the side $c$ can be chosen as a complex parameter (i.e.\ the length of $c$ and the angle $\theta$). The \enquote{unfolded} canonical cover, resembling a turtle, is pictured in \autoref{fig:C4unfolded} (\autoref{sec:intro}).

\begin{figure}
\centering
\raisebox{10pt}{\includegraphics{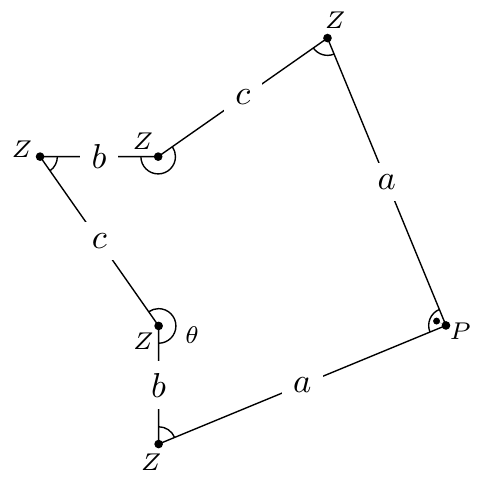}}
\quad
\includegraphics{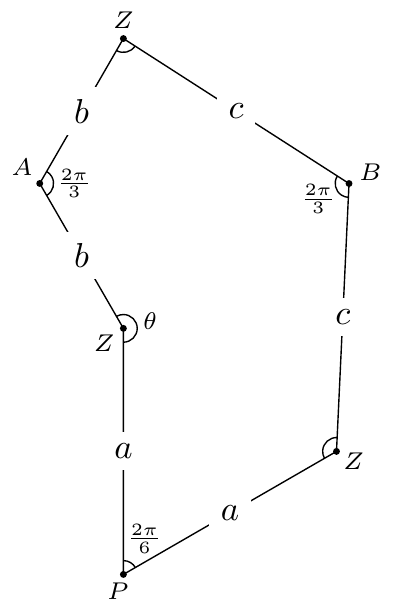}
\qquad
\raisebox{22pt}{\includegraphics{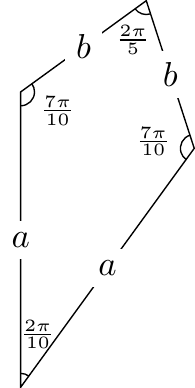}}
\caption{\textsc{Left:} A $4$-differential of genus $1$ with a zero and pole, each of order $3$. The length of $c$ and the angle $\theta$ give a complex parameter.
\textsc{Center:} A $6$-differential of genus $0$ with a single zero of order $1$, a pole of order $5$ and two poles of order $4$ each. The length of $b$ and the angle $\theta$ give a complex parameter.
\textsc{Right:} The unique (up to scale) $10$-differential on $\bP^1$ with poles of order $3$, $8$ and $9$.%
}
\label{fig:C4onE}
\label{fig:C6onP1}
\label{fig:C10onP1}
\end{figure}


\subsection*{Points of order \texorpdfstring{$3$}{3}}
Similarly, we can construct flat surfaces $(Y_{\tau},\eta_{\tau})\in \Omega\cY(6)$ admitting a six-fold symmetry $\alpha$ and a six-fold zero in terms of a parameter $\tau=\tau(b,\theta)$.

By \autoref{prop:possibleorders}, the quotient $X/\alpha$ is of genus $0$ with two fully ramified points and two points that are fully ramified over an intermediate cover of degree $3$. For the flat picture, this implies that we have a zero with angle $7\cdot 2\pi/6$, a pole with angle $2\pi/6$ and two poles with angles $2\pi/3$, see \autoref{fig:C6onP1} where the sides are identified by translation and rotation of multiplies of $2\pi/6$ to give a surface of genus $0$. Equivalently, this is a $6$-differential on $\bP^1$ with a single zero of order $1$, a pole of order $5$ and two poles of order $4$, admitting a canonical cover with only a single zero, cf.\ \cite[\S 2]{kdifferentials}.
The \enquote{unfolded} canonical cover, resembling a hurricane, is pictured in \autoref{fig:C6unfolded} (\autoref{sec:intro}).



\subsection*{Point of order \texorpdfstring{$6$}{6}}
By \autoref{thm:e2thm}, there is a unique Prym differential $(X,\omega)$ with a symmetry of order $12$ situated on $W_{12}(6)$.

By \autoref{prop:possibleorders}, we may picture this as a degree $12$ cyclic cover of $\bP^1$ with two points of order $12$ and one point of order $3$. Hence, by \cite{kdifferentials}, $(X,\omega)$ is the canonical cover of a $12$-differential on $\bP^1$ with a pole of order $5$ that pulls back to the zero, and poles of order $11$ and $8$ at the totally ramified point and the point of order $3$, respectively. Equivalently, we may glue a quadrilateral with two angles of $7\pi/12$ each and angles of $2\pi/12$ and $2\pi/3$ to give a surface of genus $0$, see \autoref{fig:C12onP1}. By \enquote{unfolding} once, i.e.\ taking the canonical $2$-cover, we obtain the $6$-differential on $\bP^1$ that exhibits $(X,\omega)$ as a fibre in the \Yfamily{} (see \autoref{fig:C12onP1}). Taking the canonical degree $3$ cover, we can cut and re-glue as indicated in \autoref{fig:C12onP1} to obtain the $4$-differential that is a $C_{12}$-eigenform on the elliptic curve with an automorphism of order $6$ in the shape of the \Xfamily{} (see~\autoref{fig:C4onE}).

\begin{figure}
\hfill\includegraphics{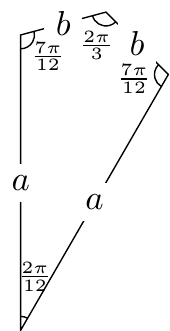}\hfill
\includegraphics{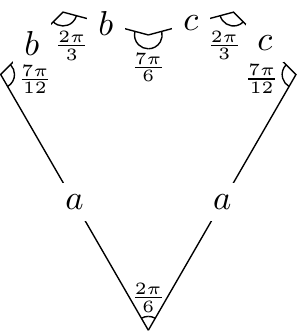}\hfill
\includegraphics{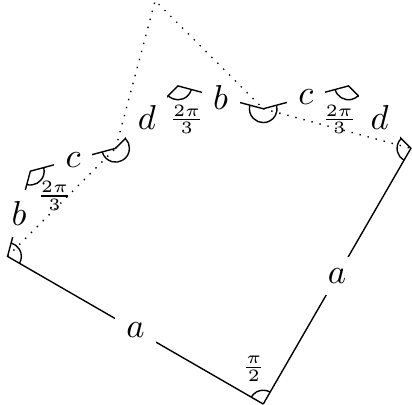}\hspace*{\fill}
\caption{\textsc{Left:} A $12$-differential of genus $0$ with a pole of order $5$ that pulls back to the zero, and poles of order $11$ and $8$. Note that this is unique up to scaling. 
\textsc{Center:} By taking the canonical double cover, we obtain a $6$-differential on $\bP^1$ as in \autoref{fig:C6onP1}. 
\textsc{Right:} By taking a triple cover, we obtain a $4$-differential on the elliptic curve $E_\zeta$ with an automorphism of order $6$. If we cut and re-glue as indicated, we obtain a polygon as in \autoref{fig:C4onE}.}
\label{fig:C12onP1}
\end{figure}

\subsection*{Point of order \texorpdfstring{$5$}{5}}
Finally, by~\autoref{thm:e5thm} there is a unique point of order $5$, i.e. an $(X,\omega)$ with a symmetry of order $10$ and a six-fold-zero differential.

More precisely, $X$ is a degree $10$ cyclic cover of $\bP^1$ ramified over three points, two of order $10$ and one of order $5$. Hence, $(X,\omega)$ is the canonical cover of a $10$-differential $\xi$ on $\bP^1$ with a pole of order $3$ (that pulls back to the single zero on $X$) and poles of order $9$ and $8$ at the second fixed point and the point of order $5$ respectively, cf.\ \cite[Prop.~2.4]{kdifferentials}. Equivalently, the flat picture has angles of size $2\pi/10$, $2\pi/5$ and two angles of size $7\cdot 2\pi/10$ each, see \autoref{fig:C10onP1} where the sides are identified by translation and rotation of multiplies of $2\pi/10$ to give a surface of genus $0$. Note that this differential is unique up to scaling. The \enquote{unfolded} canonical cover is shown in \autoref{fig:C10unfolded}.


\begin{figure}
\includegraphics{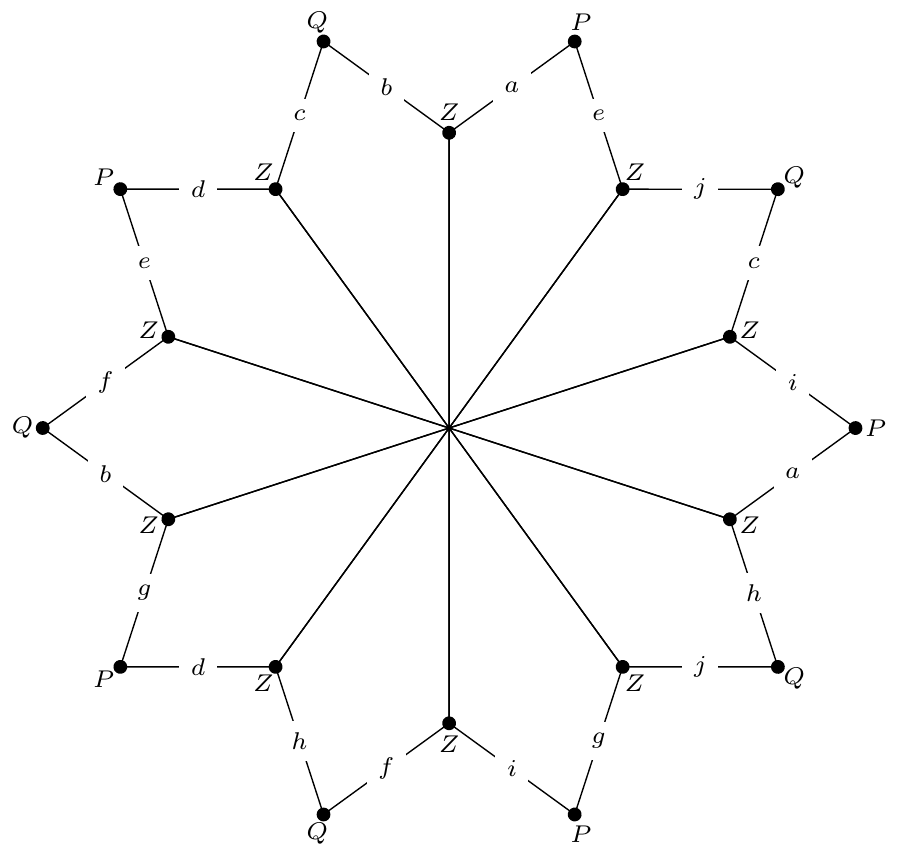}
\caption{The canonical $10$-cover of the $10$-differential in \autoref{fig:C10onP1}: a $C_{10}$-eigendifferential of genus $4$ with a single zero $Z$.}
\label{fig:C10unfolded}
\end{figure}

\section{Genus asymptotics}\label{sec:genus}

The aim of this section is to describe the asymptotic behaviour of the genus $g(\W)$ of $\W$ with respect to $D$. 

As additional boundary components make the calculation of the Euler characteristic for $D=d^2$ more tedious (cf.\ \cite[\S 13]{bainbridgeeulerchar}), we will assume throughout this section that $\W$ is primitive, i.e.\ that $D$ is not a square.

\begin{theorem}\label{genusasymptotics}
There exist constants $C_1,C_2>0$, independent of $D$, such that
\[C_1\cdot D^{3/2} < g(\W) < C_2\cdot D^{3/2}.\]
\end{theorem}
More precisely, we give the following explicit upper bound on the genus.
\begin{prop}\label{upperbound}
The genus of $\W$ satisfies $g(\W) < 1 + D^{3/2}\cdot \frac{35}{48\pi^2}.$
\end{prop}
We also give an explicit lower bound.
\begin{prop}\label{lowerbound}
The genus of $\W$ satisfies
\[g(\W)\geq \frac{3}{200} D^{3/2} - \frac{D}{6} - D^{3/4} - 150.\]
\end{prop}

\begin{cor}
The only curves $\W$ with $g(\W)=0$ are the loci for $D\leq 20$.
\end{cor}

\begin{proof}
By \autoref{lowerbound}, $g(\W)>0$ whenever $D>1050$. The smaller values of $D$ were checked by computer.
\end{proof}

Recall that the genus of $\W$, $g(\W)$, is given by $(D>12)$
\[g(\W)=h_0(\W) - \frac{\chi(\W)}{2} - \frac{C(\W)}{2} - \frac{e_2(\W)}{4} - \frac{e_3(\W)}{3},\]
where $\chi(W_D)$ is the (orbifold) Euler characteristic of $\W$, $C(\W)$ denotes the number of cusps and $e_d(\W)$ the number of points of order $d$ on $\W$. Moreover, by \cite{lanneaunguyen2}, $h_0(\W)=1$ and by \cite[Theorem 4.1]{moellerprym},
\[\chi(W_D)=-7\chi(X_D),\]
where $X_D$ is the Hilbert modular surface of discriminant $D$. Moreover, $\chi(X_D)$ was calculated, for fundamental $D$, by Siegel in terms of the Dedekind zeta function $\zeta_D$ of $\bQ(\sqrt{D})$. For non-fundamental $D$ we write $D=f^2 D_0$, where $f$ is the conductor of $D$ and $\bigl(\frac{D_0}{p}\bigr)$ for the Legendre symbol, if $p$ is a prime. Furthermore, we set
\[F(D)=\prod_{p | f}\left(1-\left(\frac{D_0}{p}\right)p^{-2}\right),\]
where the product runs over all prime divisors $p$ of $f$, and thus have
\[\chi(X_D)=\chi(X_{f^2D_0})=2f^3 \zeta_{D_0}(-1) F(D)=D^{3/2}\zeta_{D_0}(2) \frac{F(D)}{2\pi^4},\]
using the functional equation of $\zeta_{D_0}$, cf. \cite[\S 2.3]{bainbridgeeulerchar}. Finally, using Euler products, we obtain the classical bounds
\[\zeta(2)^2=\frac{\pi^4}{36}>\zeta_{D_0}(2)>\zeta(4)=\frac{\pi^4}{90}\]
and 
\[ \frac{\zeta(2)}{\zeta(4)}=\frac{15}{\pi^2} > F(D) > \frac{1}{\zeta(2)}=\frac{6}{\pi^2}.\]

We can now give an upper bound on $g(\W)$. 

\begin{proof}[Proof of \autoref{upperbound}]
As $C(\W), e_2(\W), e_3(\W)>0$, these terms may be neglected yielding
\[g(\W)\leq 1+ \frac{7}{2}\chi(X_D) < 1 + D^{3/2}\cdot \frac{35}{48\pi^2},\]
using the bounds given above.
\end{proof}

Obtaining a lower bound is slightly more involved, as it involves bounding the number of cusps and orbifold points from above. 
In general, the cusps are hardest to control, but by \cite{lanneaunguyen} and \cite{mcmTCingenustwo}, we have
\[C(W_D(6))=C(W_D(2)),\]
i.e. the Teichmüller curves of discriminant $D$ in $\cM_2$ and $\cM_4$ have the same number of cusps. Moreover, denote by $P_D$ the \emph{product locus} in $\cA_2$, i.e. abelian surfaces that are polarized products of elliptic curves. This is a union of modular curves and, again by \cite{mcmTCingenustwo},
\[C(\W)=C(P_D).\]
To bound the cusps we may therefore proceed in complete analogy to \cite[\S 6]{mukamelorbifold}.

\begin{lemma}\label{cuspbound}
The cusps are bounded from above by
\[\frac{C(\W)}{2}\leq D^{3/4} + 150 + \frac{5}{4}\chi(X_D).\]
\end{lemma}

\begin{proof}
By \cite[Theorem 2.22]{bainbridgeeulerchar}, $\chi(P_D)=-\frac{5}{2}\chi(X_D)$. Moreover, by \cite[Proposition 6.5]{mukamelorbifold}, the number of connected components of $P_D$ can be bounded by
\[h_0(P_D)\leq D^{3/4} + 150.\]
Therefore, we may write
\begin{align*}
-\frac{C(\W)}{2}&=-\frac{C(P_D)}{2}=g(P_D)-h_0(P_D) + \frac{\chi(P_D)}{2}+\sum_d\left(1-\frac{1}{d}\right)e_d(P_D)\\
&\geq -h_0(P_D)+ \frac{\chi(P_D)}{2}
\geq -D^{3/4} - 150 - \frac{5}{4}\chi(X_D),
\end{align*}
which yields the claim.
\end{proof}

Next, we must bound the number of orbifold points.

\begin{lemma}
The number of points of order $2$ satisfies $e_2(D)<\frac{D}{2}$.
\end{lemma}

\begin{proof}
By \autoref{thm:e2thm}, we have $e_2(D)\leq h(-D) + h(-\frac{D}{4})$. Now, it is well-known that class numbers of imaginary quadratic fields may be computed by counting reduced quadratic forms (cf.\ e.g.\ \cite[\S 5.3]{CohenCANT}), giving $h(-D)<\frac{D}{3}$ and thus proving the claim.
\end{proof}

\begin{lemma}\label{e3bound}
The number of points of order $3$ satisfies $e_3(D)<\frac{D}{6}$.
\end{lemma}

\begin{proof}
By \autoref{thm:e3thm},
\[e_3(D) \leq \# \{a,i,j\in\bZ : a^2 + 3j^2 + (2i-j)^2=D\}/12.\]
The integers $a$ and $j$ essentially determine $i$. Moreover, $a$ must have the same parity as $D$ giving at most $\sqrt{D}/2$ choices (up to sign) and $j$ ranges (again up to sign) over at most $\sqrt{D}$ possibilities. Accounting for sign choices and dividing by $12$ yields the claim.
\end{proof}

\begin{rem}
The bound from \autoref{e3bound} can be improved. Indeed, by the theory of modular forms of half-integral weight, integral solutions of positive definite quadratic forms can always be realized as coefficients of a suitable modular form, see \cite{ShimuraHalfIntegral} or e.g.\ \cite{lehman} for the concrete case at hand. In particular, the integral solutions of $a^2 + 3j^2 + (2i-j)^2$ are coefficients of a modular form of weight $3/2$, level $12$ and Kronecker character $12$. Hence, 
\[e_3(D) < C\cdot D^{3/4}\]
for some constant $C$ that is independent of $D$ (cf.\ e.g.\ \cite[Theorem 2.1]{moellerzagier} for growth rates of coefficients of modular forms).
\end{rem}

This permits us to also give a lower bound for $g(\W)$, proving \autoref{genusasymptotics}.

\begin{proof}[Proof of \autoref{lowerbound}]
By the above bounds, we have
\begin{align*}
g_D& =h_0(\W) - \frac{\chi(\W)}{2} - \frac{C(\W)}{2} - \frac{e_2(\W)}{4} - \frac{e_3(\W)}{3}\\
&> \frac{9}{4}\chi(X_D)-\frac{D}{6} - D^{3/4}-150,
\end{align*}
which yields the claim by the above bounds on $\chi(X_D)$.
\end{proof}

%
%
%
%
%


\begin{table}[ht]
\centering
\begin{tabular}{r  r  r  r  r  r  r  r}\toprule
  $D$ & $g$ & $\chi$  & $C$ & $e_2$ & $e_3$ & $e_5$ & $e_6$ \\\midrule
$5$&$0$&$-7/15$&$1$&$0$&$1$&1&\\
$8$&$0$&$-7/6$&$2$&$1$&$1$&&\\
$12$&$0$&$-7/3$&$3$&$1$&$0$&&1\\
$13$&$0$&$-7/3$&$3$&$0$&$2$&&\\
$17$&$0$&$-14/3$&$6$&$0$&$1$&&\\
$20$&$0$&$-14/3$&$5$&$2$&$1$&&\\
$21$&$1$&$-14/3$&$4$&$0$&$1$&&\\
$24$&$1$&$-7$&$6$&$2$&$0$&&\\
$28$&$1$&$-28/3$&$7$&$2$&$2$&&\\
$29$&$1$&$-7$&$5$&$0$&$3$&&\\
$32$&$1$&$-28/3$&$7$&$2$&$2$&&\\
$33$&$2$&$-14$&$12$&$0$&$0$&&\\
$37$&$1$&$-35/3$&$9$&$0$&$4$&&\\
$40$&$2$&$-49/3$&$12$&$2$&$2$&&\\
$41$&$3$&$-56/3$&$14$&$0$&$1$&&\\
$44$&$3$&$-49/3$&$9$&$4$&$2$&&\\
$45$&$4$&$-14$&$8$&$0$&$0$&&\\
$48$&$4$&$-56/3$&$11$&$2$&$1$&&\\
$52$&$4$&$-70/3$&$15$&$2$&$2$&&\\
$53$&$4$&$-49/3$&$7$&$0$&$5$&&\\
$56$&$6$&$-70/3$&$10$&$4$&$2$&&\\
$57$&$7$&$-98/3$&$20$&$0$&$1$&&\\
$60$&$8$&$-28$&$12$&$4$&$0$&&\\
$61$&$6$&$-77/3$&$13$&$0$&$4$&&\\
$65$&$8$&$-112/3$&$22$&$0$&$2$&&\\
$68$&$6$&$-28$&$14$&$4$&$3$&&\\
$69$&$10$&$-28$&$10$&$0$&$0$&&\\
$72$&$10$&$-35$&$16$&$2$&$0$&&\\
$73$&$10$&$-154/3$&$32$&$0$&$2$&&\\
$76$&$11$&$-133/3$&$21$&$4$&$2$&&\\
$77$&$9$&$-28$&$8$&$0$&$6$&&\\
$80$&$10$&$-112/3$&$16$&$4$&$2$&&\\
$84$&$14$&$-140/3$&$18$&$4$&$1$&&\\
$85$&$12$&$-42$&$16$&$0$&$6$&&\\
$88$&$15$&$-161/3$&$22$&$2$&$4$&&\\
$89$&$17$&$-182/3$&$28$&$0$&$1$&&\\
$92$&$15$&$-140/3$&$13$&$6$&$4$&&\\
$93$&$15$&$-42$&$12$&$0$&$3$&&\\
$96$&$18$&$-56$&$20$&$4$&$0$&&\\
$97$&$21$&$-238/3$&$38$&$0$&$2$&&\\
$101$&$14$&$-133/3$&$15$&$0$&$5$&&\\
$104$&$18$&$-175/3$&$20$&$6$&$2$&&\\
$105$&$27$&$-84$&$32$&$0$&$0$&&\\
\bottomrule
\end{tabular}
\ 
\begin{tabular}{r  r  r  r  r  r}\toprule
  $D$ & $g$ & $\chi$  & $C$ & $e_2$ & $e_3$ \\\midrule
$108$&$21$&$-63$&$21$&$4$&$0$\\
$109$&$18$&$-63$&$25$&$0$&$6$\\
$112$&$22$&$-224/3$&$29$&$2$&$4$\\
$113$&$26$&$-84$&$32$&$0$&$3$\\
$116$&$21$&$-70$&$25$&$6$&$3$\\
$117$&$21$&$-56$&$16$&$0$&$0$\\
$120$&$29$&$-238/3$&$20$&$4$&$2$\\
$124$&$31$&$-280/3$&$29$&$6$&$2$\\
$125$&$21$&$-175/3$&$15$&$0$&$5$\\
$128$&$25$&$-224/3$&$22$&$4$&$4$\\
$129$&$37$&$-350/3$&$44$&$0$&$1$\\
$132$&$29$&$-84$&$26$&$4$&$0$\\
$133$&$27$&$-238/3$&$22$&$0$&$8$\\
$136$&$35$&$-322/3$&$36$&$4$&$2$\\
$137$&$37$&$-112$&$38$&$0$&$3$\\
$140$&$33$&$-266/3$&$18$&$8$&$4$\\
$141$&$34$&$-84$&$18$&$0$&$0$\\
$145$&$46$&$-448/3$&$58$&$0$&$2$\\
$148$&$39$&$-350/3$&$37$&$2$&$4$\\
$149$&$30$&$-245/3$&$19$&$0$&$7$\\
$152$&$37$&$-287/3$&$18$&$6$&$4$\\
$153$&$45$&$-140$&$52$&$0$&$0$\\
$156$&$46$&$-364/3$&$26$&$8$&$2$\\
$157$&$36$&$-301/3$&$25$&$0$&$8$\\
$160$&$44$&$-392/3$&$40$&$4$&$4$\\
$161$&$55$&$-448/3$&$40$&$0$&$2$\\
$164$&$37$&$-112$&$34$&$8$&$3$\\
$165$&$42$&$-308/3$&$18$&$0$&$4$\\
$168$&$51$&$-126$&$24$&$4$&$0$\\
$172$&$53$&$-147$&$37$&$4$&$6$\\
$173$&$37$&$-91$&$13$&$0$&$9$\\
$176$&$49$&$-392/3$&$29$&$6$&$4$\\
$177$&$66$&$-182$&$52$&$0$&$0$\\
$180$&$52$&$-140$&$36$&$4$&$0$\\
$181$&$49$&$-133$&$33$&$0$&$6$\\
$184$&$66$&$-518/3$&$38$&$4$&$4$\\
$185$&$66$&$-532/3$&$46$&$0$&$2$\\
$188$&$53$&$-392/3$&$19$&$10$&$4$\\
$189$&$51$&$-126$&$26$&$0$&$0$\\
$192$&$57$&$-448/3$&$34$&$4$&$2$\\
$193$&$77$&$-686/3$&$74$&$0$&$4$\\
$197$&$44$&$-343/3$&$21$&$0$&$11$\\
$200$&$56$&$-455/3$&$36$&$6$&$4$\\
\bottomrule
\end{tabular}
\caption{Topological invariants of the Teichmüller curves $W_D(6)$ for nonsquare discriminant. The number of cusps is described in \cite{lanneaunguyen}, the Euler characteristic in \cite{moellerprym}.}\label{thetable}
\end{table}

\FloatBarrier 

\emergencystretch=3em
\printbibliography

\end{document}